\documentclass{article}
\usepackage{amsmath}
\usepackage{amsthm,amssymb,enumerate,hyperref}
\usepackage[a4paper,margin=2cm,centering,nohead,ignorefoot,footskip=5ex]{geometry}
\usepackage{url}
\usepackage{xcolor}
\usepackage{mathabx}
\hypersetup{
    colorlinks,
    linkcolor={red!50!black},
    citecolor={red!50!black},
    urlcolor={blue!80!black}
}
\theoremstyle{plain}
\newtheorem{theorem}{Theorem}[section]
\newtheorem{proposition}[theorem]{Proposition}
\newtheorem{notation}[theorem]{Notation}
\newtheorem{lemma}[theorem]{Lemma}
\theoremstyle{definition}
\newtheorem{remark}[theorem]{Remark}
\newtheorem{definition}[theorem]{Definition}
\newtheorem{corollary}[theorem]{Corollary}
\newtheorem{example}[theorem]{Example}

\newenvironment{rcases}
{\left.\begin{aligned}}
	{\end{aligned}\right\rbrace}

\newcommand{\N}{\mathbb{N}}
\newcommand{\R}{\mathbb{R}}

\newcommand{\eps}{\varepsilon}
\newcommand{\norm}[1]{\left\lVert#1\right\rVert}

\newcommand{\Fix}{\mathrm{Fix}\,}
\newcommand{\proj}{\textup{P}}
\newcommand{\Id}{\mathrm{Id}}
\newcommand{\CAT}{\mathrm{CAT}(0)}
\newcommand{\SE}[1]{\begin{equation*}\begin{split}#1\end{split}\end{equation*}}

\begin{document}

\title{Strong convergence for the alternating Halpern-Mann iteration in $\CAT$ spaces  }

\author{Bruno Dinis${}^{a}$ and Pedro Pinto${}^{b}$\\[2mm]
	\footnotesize ${}^{a}$ Escola de Ci\^encias e Tecnologia da
	Universidade de Évora,\\ 
	\footnotesize Rua Romão Ramalho, 59, 7000-671~Évora, Portugal\\
	\footnotesize E-mail:  \protect\url{bruno.dinis@uevora.pt}\\[ 2mm]
	\footnotesize ${}^{b}$ Department of Mathematics, Technische Universit{\"a}t Darmstadt,\\ 
	\footnotesize Schlossgartenstrasse 7, 64289 Darmstadt, Germany \\
	\footnotesize E-mail:  \protect\url{pinto@mathematik.tu-darmstadt.de}
}
\maketitle

\begin{abstract}
In this paper we consider, in the general context of $\CAT$ spaces, an iterative schema which alternates between Halpern and Krasnoselskii-Mann style iterations. We prove, under suitable conditions, the strong convergence of this algorithm, benefiting from ideas from the proof mining program. We give quantitative information in the form of effective rates of asymptotic regularity and of metastability (in the sense of Tao).
Motivated by these results we are also able to obtain strongly convergent versions of the forward-backward and the Douglas-Rachford algorithms.
Our results generalize recent work by Bo\c{t}, Csetnek and  Meier, and Cheval and Leu\c{s}tean.
\end{abstract}
\section{Introduction}

In the context of $\CAT$ spaces, we show the strong convergence of an iterative schema which alternates between Halpern and Krasnoselskii-Mann style iterations, while also obtaining quantitative information in the form of rates of asymptotic regularity and rates of metastability.

Let $H$ be a Hilbert space and $C$ a nonempty closed convex subset. A mapping $U:C \to C$ is said to be \emph{nonexpansive} on $C$ if for all  $x, y \in C$ one has  $\norm{U (x) -U (y)}\leq \norm{x -y}$. A well-known method for finding fixed points of a nonexpansive mapping $U:C \to C$ is 
\begin{equation}\label{e:KM}\tag{KM}
x_{n+1} = (1 - \beta_n)U(x_n) + \beta_nx_n,
\end{equation}
with $x_0 \in C$ a starting point and $(\beta_n) \subset [0,1]$. This iteration was introduced by Mann \cite{Mann(53)} and Krasnoselskii \cite{Kraso(55)}. Even though this iteration has the propitious property of being Fejér monotone with respect to the fixed point set, in general it only converges weakly to a fixed point (see e.g.\ \cite[Theorem~2]{R(79)} and \cite[Corollary~5.2 and Remark~5.3]{BHMR(04)}). This motivated several modifications of \eqref{e:KM} in order to ensure strong convergence. One such version was introduced by Yao, Zho and Liou \cite{YZY}, and rediscovered recently by
  Bo\c{t}, Csetnek, and  Meier~\cite{Botetal(19)}, where strong convergence is guaranteed by introducing Tikhonov regularization terms $(\gamma_n)\subset [0,1]$
\begin{equation}\label{e:TKM}\tag{T-KM}
x_{n+1} = (1 - \beta_n)U(\gamma_nx_n) + \beta_n(\gamma_nx_n).
\end{equation}
 In \cite[Theorem~3]{Botetal(19)} Bo\c{t} \emph{et al.} showed that $(x_n)$ generated by \eqref{e:TKM} converges strongly to the fixed point of $U$ with minimum norm $P_{\Fix(U)}(0)$, provided that the following conditions hold
\begin{enumerate}[$(i)$]
\item $(\gamma_n) \subset (0,1]$, \, $\lim \gamma_n =1$,\, $\sum_{n \geq 0} (1-\gamma_n)=\infty$,\, $\sum_{n \geq 0}|\gamma_{n+1}-\gamma_{n}|< \infty$,
\item $(\beta_n)\subset [0,1)$, \, $\sum_{n \geq 0}|\beta_{n+1}-\beta_{n}|< \infty$, and  $\limsup \beta_n <1$.
\end{enumerate}
From this result,  Bo\c{t} \emph{et al.} introduced strongly convergent versions of the forward-backward and the Douglas-Rachford algorithms \cite[Theorems~7 and 10]{Botetal(19)}, which are well-known splitting methods that weakly approximate zeros of
a sum of monotone operators.

A different iteration was introduced by Halpern in \cite{Halpern67}
\begin{equation}\label{e:Halpern}\tag{H}
x_{n+1}=(1 - \alpha_n)T(x_n) + \alpha_nu,
\end{equation}
where $T:C \to C$ is a nonexpansive map, $x_0,u \in C$ and $(\alpha_n) \subset [0,1]$. Halpern showed the strong convergence of \eqref{e:Halpern} to the
metric projection onto the set $\Fix(T)$ of the anchor point $u$, in Hilbert spaces. The conditions considered by Halpern prevented the natural choice $\alpha_n =\frac{1}{n+1}$, which was later overcome by Wittmann \cite{Wittmann(92)}. The strong convergence of \eqref{e:Halpern} was first extended to the setting of uniformly smooth Banach spaces by Reich in \cite{R(80)}. In \cite{X(02),Xu(04)}, Xu established strong convergence, also in uniformly smooth Banach spaces, under conditions which are incomparable to those considered by Wittmann but which still allowed $\alpha_n =\frac{1}{n+1}$.

As observed by Cheval and Leu\c{s}tean in \cite{CK(ta)},  the iteration \eqref{e:TKM} can be seen as alternating between a Mann style construction and a special case of the Halpern schema. Indeed, the iteration considered by Cheval and Leu\c{s}tean can be written in the following way
\begin{equation}\label{CL}\tag{CL}
\begin{cases}
x_{2n+1}&=(1-\alpha_n) x_{2n}+\alpha_nu\\
x_{2n+2}&=(1-\beta_n)U(x_{2n+1})+\beta_nx_{2n+1},
\end{cases}
\end{equation}
which corresponds to \eqref{e:TKM} when $u=0$. They considered this iteration in the general setting of hyperbolic spaces and obtained rates of asymptotic regularity. Yet no strong convergence results were established. Moreover, it was shown recently in \cite{CKL(ta)} that this iteration is essentially a modified Halpern iteration, as introduced in \cite{KX}, and studied in the setting of $\CAT$ spaces in \cite{CP}. Rates of asymptotic regularity and of metastability for the modified Halpern iteration were obtained in \cite{SK}, and shown to be transferable to the iteration \eqref{CL} (and \emph{vice-versa}) in \cite{CKL(ta)}.

Let $X$ be a $\CAT$ space and $C$ a nonempty convex closed subset. In this paper we will consider an iteration $(x_n)$ generated by the recursive schema 
\begin{equation}\label{e:MannHalpern}\tag{HM}
	\begin{cases}
		x_{2n+1}&=(1-\alpha_n)T(x_{2n})\oplus \alpha_n u\\
		x_{2n+2}&=(1-\beta_n)U(x_{2n+1})\oplus\beta_nx_{2n+1},
	\end{cases}
\end{equation}
where $x_0, u\in C$, $(\alpha_n), (\beta_n)\subset [0,1]$ are sequences of real numbers  and $T,U:C \to C$ are nonexpansive mappings. The symbol ``$\oplus$'' denotes the application of the convex operator $W$ (see Section~\ref{s:preliminaries} for details), which is the usual linear convex combination in the case of normed spaces. We show the following result:
\begin{theorem}\label{t:main}
	Let $X$ be a complete $\CAT$ space and $C \subseteq X$ a nonempty convex closed subset. Consider nonexpansive mappings $T,U:C \to C$ such that  $F:=\Fix(T) \cap \Fix(U) \neq \emptyset$ and $u,x_0 \in C$.  Assume that  $(\alpha_n) \subset [0,1]$, $(\beta_n) \subset (0,1)$ are sequences of real numbers satisfying
		\begin{align*}
			&(i)\ \lim \alpha_n =0; \qquad (ii)\ \sum_{n \geq 0} \alpha_n= \infty; \qquad (iii)\ \sum_{n \geq 0} |\alpha_{n+1} - \alpha_{n}|< \infty;\\
			 &(iv)\ \sum_{n \geq 0} |\beta_{n+1} - \beta_{n}|< \infty; \qquad (v)\ 0< \liminf \beta_n \leq \limsup \beta_n <1.
			 \end{align*}
	Then $(x_n)$ generated by \eqref{e:MannHalpern} converges strongly to $\proj_F(u)$. 
\end{theorem}
Theorem~\ref{t:main} is a twofold generalization of \cite[Theorem~3]{Botetal(19)}. On the one hand, our iteration has \eqref{e:TKM} as a particular case. On the other hand, our strong convergence result is established in $\CAT$ spaces -- frequently considered the non-linear generalization of Hilbert spaces. The \eqref{e:MannHalpern} iteration also extends the one in \cite{CK(ta)} -- which is the particular case when $T= \Id_C$. Even though the context of \cite{CK(ta)} is more general than that of  $\CAT$ spaces, Theorem~\ref{t:main} proves strong convergence while  \cite[Theorems~4.1 and 4.2]{CK(ta)} only establish asymptotic regularity for the iteration. This greater generality is achieved at the cost of additionally assuming the mild condition $0< \liminf \beta_n$. Similarly to the strongly convergent versions of the forward-backward and the Douglas-Rachford algorithms introduced in \cite{Botetal(19)}, our generalized iteration is also used to define extended versions of these algorithms. 

As an initial step into proving our main theorem we first establish quantitative results in the spirit of the proof mining program \cite{K(08),K(18)}. In this framework one looks to obtain quantitative information, such as rates of convergence and rates of metastability, guided by proof-theoretical techniques.  For example, while in general it is not possible to obtain computable information on the value of $n$ in the Cauchy property (see e.g.\ \cite{N(15)}), 
 \begin{equation*}
\forall \eps >0\,\exists n \in \N\, \forall i,j \geq n  \left(d(x_i,x_j)\leq \eps \right),
\end{equation*}
in many cases, logical results guarantee quantitative information for the equivalent finitary version
 \begin{equation}\label{metastab}\tag{$\dagger$}
\forall \eps >0 \,\forall f : \N \to \N \,\exists n \, \forall i,j \in [n,f(n)] \left(d(x_i,x_j)\leq \eps \right),
\end{equation}
where $[n,f(n)]$ denotes the set $\{n,n+1, \cdots, f(n)\}$.
The property \eqref{metastab} was popularized by Tao \cite{T(08b),T(08a)} under the name \emph{metastability}. Quantitative information for \eqref{metastab} takes the form of a \emph{rate of metastability}, i.e.\ a computable functional $\mu: (0, +\infty) \times \N^\N \to \N$ satisfying
 \begin{equation*}\label{metastabb}
\forall \eps >0\,\forall f : \N \to \N \,\exists n \leq \mu(\eps,f) \, \forall i,j \in [n,f(n)] \left(d(x_i,x_j)\leq \eps \right).
\end{equation*}
Note that a rate of metastability for \eqref{metastab} does not entail computable information for the equivalent Cauchy property. Even though these two notions are equivalent, this equivalence is non-effective as the proof is by contradiction. 
Still, metastability results have a far reaching scope as can be seen in several recent results, see for example \cite{K(05)i,AGT(10),GT(08),KL(09),T(08b)}. Another striking feature of proof mining is the possibility to identify precisely the required conditions needed to prove a result. This frequently allows to single out unused hypothesis and thus to obtain generalized results. 
In this sense, Theorem~\ref{t:main} follows from a generalization to a non-linear setting of a corresponding quantitative result in Hilbert spaces. In Hilbert spaces the result can be established via a sequential weak compactness argument (see Section~\ref{s:final}). It is not clear how to carry out such an argument in $\CAT$ spaces. Nevertheless, the technique developed in \cite{FFLLPP(19)} shows that, through a quantitative treatment, it is possible to bypass this argument in Hilbert spaces (as was done e.g.\ in \cite{DP(ta),DP(21)}). Extending such quantitative results to the setting of $\CAT$ spaces allows to conclude the main theorem. These quantitative results may also be seen as a natural continuation of previous quantitative analyses  (see e.g.\ \cite{DP(ta),KL(12),DP(MAR)}).
Even though our results and proofs are inspired by proof theoretical techniques,  these are only used as an intermediate step and are not visible in the final product. As such we do not presuppose any particular knowledge of logical tools.

The structure of the paper is the following. Some relevant terminology and useful lemmas are recalled in Section~\ref{s:preliminaries}. The central results  are obtained in Sections~\ref{s:asymptoticregularity} and \ref{s:convergence} where we show that a sequence $(x_n)$ generated by \eqref{e:MannHalpern} is asymptotically regular (Section~\ref{s:asymptoticregularity}) and has the metastability property (Section~\ref{s:convergence}), while also obtaining the corresponding quantitative information. In Section~\ref{s:Projection} we show a quantitative version of the metric projection in $\CAT$ spaces which is used to establish the metastability of the sequence. We show how Theorem~\ref{t:main} follows from the metastability property in Section~\ref{s:Strongconvergence}. Afterwards,  in Section~\ref{s:FBDR}, we study generalized versions of the forward-backward and the Douglas-Rachford algorithms (now in the context of Hilbert spaces). Some final remarks are left for Section~\ref{s:final}.

\section{Preliminaries}\label{s:preliminaries}

\subsection{Hyperbolic and $\CAT$ spaces}
Consider a triple $(X,d,W)$ where $(X,d)$ is a metric space and $W:X \times X \times [0,1] \to X$ is a function satisfying, for all $x,y,z,w\in X$ and $\lambda, \lambda' \in [0,1]$,
\begin{enumerate}
	\item[(W1)] $d(z, W(x, y,\lambda))\leq (1-\lambda)d(z, x) + \lambda d(z, y)$
	\item[(W2)] $d(W(x, y, \lambda), W(x, y, \lambda'))=|\lambda-\lambda'|d(x, y)$
	\item[(W3)] $W(x,y,\lambda)=W(y,x, 1-\lambda)$
	\item[(W4)] $d(W(x,y,\lambda), W(z,w,\lambda))\leq (1-\lambda)d(x,z)+\lambda d(y,w)$.
\end{enumerate}
A triple in the conditions above is called a \emph{hyperbolic space}. This formulation was introduced by Kohlenbach in \cite{K(05)} under the name of $W$-hyperbolic spaces. The convexity function $W$ was first considered by Takahashi in \cite{T(70)}, where a triple $(X,d,W)$ satisfying (W1) was called a convex metric space. The notion of hyperbolic space considered here is more general than the hyperbolic spaces in the sense of Reich and Shafrir \cite{RS(90)}, and slightly more restrictive than the notion of space of hyperbolic type by Goebel and Kirk \cite{GK(83)}.

If $x,y\in X$ and $\lambda \in[0,1]$, we use the notation $(1-\lambda)x \oplus \lambda y$ to denote $W(x,y,\lambda)$. It is easy to see using (W1) that
\begin{equation}\label{e:equality}
d(x, (1-\lambda)x \oplus \lambda y)=\lambda d(x,y) \,\mbox{ and }\, d(y, (1-\lambda)x \oplus \lambda y)=(1-\lambda)d(x,y).
\end{equation}
The metric segment with endpoinds $x,y \in X$ is the set $\{(1-\lambda)x\oplus \lambda y:\lambda \in [0,1]\}$ and is denoted by $[x,y]$. A nonempty subset $C \subseteq X$ is said to be \emph{convex} if $\forall x,y \in C ([x,y] \subseteq C)$.

The class of hyperbolic spaces includes the normed spaces and their convex subsets (with $(1-\lambda)x\oplus \lambda y=(1-\lambda)x + \lambda y$, the usual convex linear combination), the Hilbert ball \cite{GR(84)} and the $\CAT$ spaces. 
The important class of $\CAT$ spaces (introduced by Alexandrov in \cite{A(51)}, and named as such by Gromov in \cite{G(87)}; see \cite{BH(13)} for a detailed treatment) is characterized as the hyperbolic spaces that satisfy the property $\mathrm{CN}^-$ (which in the presence of the other axioms is equivalent to the Bruhat-Tits $\mathrm{CN}$-inequality \cite{BT(72)}, but contrary to the latter is purely universal\footnote{I.e.\ it can can be formally restated as $\forall x_1 \dots \forall x_n A_0(x_1,\dots,x_n)$, where $A_0$ is quantifier-free/decidable. }):
\begin{equation}\tag{CN$^-$}
	\forall x,y,z \in X\, \left( d^2\left(z, \frac{1}{2}x \oplus \frac{1}{2}y\right) \leq \frac{1}{2}d^2(z, x)+\frac{1}{2}d^2(z, y)-\frac{1}{4}d^2(x, y) \right)
\end{equation}
This inequality actually extends beyond midpoints -- see e.g. \cite[Lemma 2.5]{DP(08)}: for all $x,y,z\in X$ and $\lambda \in[0,1]$
\begin{equation}\tag{CN$^+$}\label{CN}
	d^2\left(z, (1-\lambda)x \oplus \lambda y\right) \leq (1-\lambda)d^2(z, x)+\lambda d^2(z, y)-\lambda(1-\lambda)d^2(x, y).
\end{equation}
In the same way as hyperbolic spaces are considered the non-linear counterpart of normed spaces, $\CAT$ spaces are the non-linear generalizations of Hilbert spaces.

As shown by Leu\c{s}tean \cite{L(07)}, $\CAT$ spaces are uniformly convex with a quadratic modulus of uniform convexity:
\begin{lemma}[{\cite[Proposition 8]{L(07)}}]\label{l:CATconvexity}
Every $\CAT$ space is a uniformly convex space and $\eta(\eps)=\frac{\eps^2}{8}$ is a modulus of uniform convexity, i.e. for all $\eps\in(0,2]$, $r >0$, and $x,y,a\in X$
\begin{equation*}
	\begin{rcases}
		d(x, a)&\leq r\\
		d(y,a)&\leq r\\
		d(x,y) &\geq \eps r\\
	\end{rcases} \rightarrow d\left(\frac{1}{2}x \oplus \frac{1}{2}y, a\right)\leq (1-\eta(\eps))r.
\end{equation*}
\end{lemma}

In \cite[Proposition~14]{BN(08)} it was shown that in every metric space there exists a unique function $\langle \cdot, \cdot \rangle: X^2 \times X^2 \to \R$
satisfying
\begin{enumerate}[$(i)$]
\item $\langle \overrightarrow{xy},\overrightarrow{xy}\rangle=d^2(x,y)$
\item $\langle \overrightarrow{xy},\overrightarrow{uv}\rangle=\langle \overrightarrow{uv},\overrightarrow{xy}\rangle$
\item $\langle \overrightarrow{yx},\overrightarrow{uv}\rangle=-\langle \overrightarrow{xy},\overrightarrow{uv}\rangle$
\item $\langle \overrightarrow{xy},\overrightarrow{uv}\rangle+ \langle \overrightarrow{xy},\overrightarrow{vw}\rangle= \langle \overrightarrow{xy},\overrightarrow{uw}\rangle$,
\end{enumerate}
where $\overrightarrow{xy}$ denotes the pair $(x,y)$.
This unique function, called the \emph{quasi-linearization function}, is defined for any $(x,y),(u,v)\in X^2$ by
\begin{equation*}
\langle \overrightarrow{xy},\overrightarrow{uv}\rangle:= \frac{1}{2}\left(d^2(x,v)+d^2(y,u)-d^2(x,u)-d^2(y,v) \right). 
\end{equation*}
In $\CAT$ spaces the quasi-linearization function satisfies the Cauchy-Schwarz inequality,
\begin{equation}\label{e:CS}
\langle \overrightarrow{xy},\overrightarrow{uv}\rangle \leq d(x,y)d(u,v).
\end{equation}

We also have the following useful inequality.
\begin{lemma}\label{l:binomial}
Let $X$ be a $\CAT$ space. For every $x,y,z \in X$ and $t \in [0,1]$
\begin{equation*}
d^2((1-t)x\oplus t y,z)\leq (1-t)^2d^2(x,z)+2t(1-t)\langle \overrightarrow{xz},\overrightarrow{yz}\rangle+t^2 d^2(y,z).
\end{equation*}
\end{lemma}
\begin{proof}
Using \eqref{CN} we derive
\SE{d^2((1-t)x\oplus t y,z)&\leq (1-t)d^2(x,z)+td^2(y,z)-t(1-t)d^2(x,y)\\
&= (1-t)^2d^2(x,z)+t(1-t)\left(d^2(x,z)+d^2(y,z)-d^2(x,y) \right)+t^2d^2(y,z)\\
&=(1-t)^2d^2(x,z)+2t(1-t)\langle \overrightarrow{xz},\overrightarrow{yz}\rangle+t^2 d^2(y,z).\qedhere
}
\end{proof}

\subsection{Quantitative notions}

\begin{definition}\label{d:RN}
Let $(a_n)$ be a sequence of real numbers. 
\begin{enumerate}[$(i)$]
\item A \emph{rate of convergence} for 
$\lim a_n = 0$ is a function $\gamma:(0,+\infty)\to\N$ such that
\[\forall \eps >0\, \forall n\geq \gamma(\eps)\,\left(|a_n|\leq \eps\right).\]
%
%
%
\item A \emph{rate of divergence} for $\lim a_n =+\infty$ is a function $\gamma:\N\to\N$ such that
\[\forall k\in\N\, \forall n\geq \gamma(k)  \left(a_n \geq k\right).\]
\item A \emph{Cauchy rate} for $(a_n)$ is a function $\gamma:(0,+\infty)\to\N$ such that
\[\forall \eps >0\, \forall i,j \geq \gamma(\eps) \left(|a_{i}-a_{j}| \leq \eps \right).\]
When $a_n= \sum_{i=0}^{n} b_i$, for some sequence $(b_n) \subset [0, \infty)$, $\gamma$ is a \emph{Cauchy rate} for $(a_n)$ if $$\forall \eps >0 \, \forall n \in \N \left( \sum_{i= \gamma(\eps)+1}^{n}b_i \leq \eps \right). $$%
\end{enumerate}
\end{definition}
\begin{definition}\label{d:HS}
Let $(x_n)$ be a sequence in a metric space $(X,d)$ and $x \in X$.
\begin{enumerate}[$(i)$]
\item A \emph{rate of convergence} for $\lim x_n = x$ is a rate of convergence for $\lim d(x_n,x)= 0$. 
%
%
\item A \emph{Cauchy rate} for $(x_n)$ is a function $\gamma:(0,+\infty)\to\N$ such that
\[\forall \eps>0 \, \forall i,j \geq \gamma(\eps)\left(d(x_i,x_j)\leq \eps\right).\]
\item A \emph{rate of metastability} for $(x_n)$ is a functional $\Gamma:(0,+\infty) \times \N^{\N}\to\N$ such that
 \begin{equation*}
\forall \eps>0 \,\forall f : \N \to \N \,\exists n \leq \Gamma(\eps,f) \, \forall i,j \in [n,f(n)] \left(d(x_i,x_j)\leq \eps \right).
\end{equation*}
\end{enumerate}
\end{definition}
We say that a function $f:\N\to\N$ is \emph{monotone} if
\[
k\leq k' \to f(k)\leq f(k').
\]
Without loss of generality, we can always assume to have this property, since if needed we can replace the function $f$ with the function $f^{\max}:\N\to\N$ defined by $f^{\max}(k):=\max\{f(k')\, : \, k'\leq k\}$ for all $k\in\N$. We say that a function $g: (0, +\infty)\to \N$ is \emph{monotone} if
\[
0\leq \eps \leq \eps' \to g(\eps) \geq g(\eps').
\] 
This notion follows from a proof mining treatment where one works with a function $g':\N\to \N$ satisfying
\[
g(\eps)=g'(k)\, \text{ with } k=\lceil \eps^{-1}\rceil.
\]
\begin{lemma}\label{l:metaCauchy}
Let $(X,d)$ be a metric space and $(x_n)$ be a sequence in $X$. Then $(x_n)$ has the metastability property, i.e.\
\begin{equation*}
\forall \eps>0 \,\forall f : \N \to \N \,\exists n \in \N \, \forall i,j \in [n,f(n)] \left(d(x_i,x_j)\leq \eps \right)
\end{equation*}
if and only if $(x_n)$ is a Cauchy sequence.
\end{lemma}
\begin{proof}
If $(x_n)$ is a Cauchy sequence then it clearly has the metastability property. Assume that $(x_n)$ is not a Cauchy sequence. Then, there exists $\eps_0 >0$ such that for every $n \in \N$
\begin{equation*}
\exists i_n > j_n \geq n \left(d(x_{i_n},x_{j_n})>\eps_0 \right).
\end{equation*}
Hence, the metastability property fails for $\eps_0$ and the function $f$ defined by $f(n):=i_n$, for all $n \in \N$.
\end{proof}
Note that the proof that metastability implies the Cauchy property is non-effective. As a consequence, the existence of a computable rate of metastability does not entail a computable Cauchy rate. On the other hand, a function $\gamma$ is a Cauchy rate if and only if $\Gamma(\eps,f):=\gamma(\eps)$ is a rate of metastability (see e.g.\ \cite[Proposition~2.6]{KP(22)}). We now recall the notion of asymptotic regularity (cf. \cite{BP(66)}), and corresponding quantitative notions of rate of asymptotic regularity.
\begin{definition}
Let $(x_n)$ be a sequence in a metric space $(X,d)$ and consider a mapping $T: X \to X$. 
\begin{enumerate}[$(i)$]
\item The sequence $(x_n)$ is \emph{asymptotically regular} if $\lim d(x_{n+1},x_n)=  0$.  A \emph{rate of asymptotic regularity} for $(x_n)$ is a rate of convergence for $\lim d(x_{n+1},x_n)= 0$.

\item The sequence $(x_n)$ is \emph{asymptotically regular} with respect to $T$ if $\lim d(T(x_{n}),x_n)=  0$.  A \emph{rate of asymptotic regularity} for $(x_n)$ with respect to $T$  is a rate of convergence for $\lim d(T(x_{n}),x_n)= 0$.
\end{enumerate}
\end{definition}

\subsection{Useful lemmas}

We recall the following well-known result due to Xu.

\begin{lemma}[\cite{X(02)}]\label{L:Xu}
Let $(a_n) \subset (0,1)$ and $(r_n),(v_n)$ be real sequences such that
\begin{equation*}
(i) \sum a_n = \infty; \qquad  (ii) \limsup r_n \leq 0; \qquad (iii)\sum v_n<\infty.
\end{equation*}
 Let $(s_n)$ be a non-negative real sequence satisfying $s_{n+1}\leq (1-a_n)s_n+ a_nr_n+v_n$, for all $n \in \N$. Then $\lim s_n = 0$.
\end{lemma}
Xu's lemma has received several quantitative analyses. The next two results are (essentially) from \cite{LLPP(21),PP(ta)}.

\begin{notation}\hfill
\begin{enumerate}
\item Throughout this paper $\lceil x\rceil$ is defined as $\max\{0, \lceil x\rceil\}$
with the usual definition of $\lceil\cdot\rceil$ in the latter.
\item Consider a function $\varphi$ on tuples of variables $\bar{x}$, $\bar{y}$. If we wish to consider the variables $\bar{x}$ as parameters we write $\varphi[\bar{x}](\bar{y})$. For simplicity of notation we may then even omit the parameters and simply write $\varphi(\bar{y})$.
\end{enumerate}
\end{notation}
\begin{lemma}\label{L:xu_seq_reals_qt1}
	Let $(s_n)$ be a bounded sequence of non-negative real numbers and $D\in\N\setminus\{0\}$ an upper bound on $(s_n)$. Consider sequences of real numbers $(a_n)\subset\,[0,1]$, $(r_n)\subset \R$ and $(v_n)\subset \R^+_0$ and functions ${\rm A}:\N \to \N$ and ${\rm R}$, ${\rm V}: (0,+\infty)\to \N$ such that
	\begin{enumerate}[$(i)$]
		\item ${\rm A}$ is a rate of divergence for $\left(\sum a_n\right)$,
		\item ${\rm R}$ is such that $\forall \eps >0\, \forall n\geq {\rm R}(\eps) \, \left( r_n \leq \eps \right)$,
		\item ${\rm V}$ is a Cauchy rate for $\left(\sum v_n\right)$.
	\end{enumerate}
	Assume that for all $n\in \N$, $s_{n+1}\leq (1-a_n)s_n+a_nr_n + v_n$. Then $\lim s_n =0$ with rate of convergence
		\[\theta(\eps):=\theta[{\rm A}, {\rm R}, {\rm V}, D](\eps):={\rm A}\left(K+\left\lceil \ln\left(\frac{3D}{\eps}\right)\right\rceil\right)+1,\, \text{ with }\, K:=\max\left\{ {\rm R}\left(\frac{\eps}{3}\right), {\rm V}\left(\frac{\eps}{3}\right)+1 \right\}.
		\]
		Moreover,
		\begin{enumerate}[$(1)$]
		\item If $r_n\equiv 0$, then the function $\theta$ is simplified to
		\[
		\theta(\eps):=\widehat{\theta}[{\rm A},{\rm V},D](\eps):={\rm A}\left({\rm V}\left(\frac{\eps}{2}\right)+\left\lceil \ln\left(\frac{2D}{\eps}\right)\right\rceil+1\right)+1.
		\]
		\item If $v_n\equiv 0$, then the function $\theta$ is simplified to
		\[
		\theta(\eps):=\widecheck{\theta}[{\rm A},{\rm R},D](\eps):={\rm A}\left({\rm R}\left(\frac{\eps}{2}\right)+\left\lceil \ln\left(\frac{2D}{\eps}\right)\right\rceil\right)+1.
		\]
	\end{enumerate}
\end{lemma}

Instead of considering $\sum a_n = \infty$, one can work with the following equivalent condition
\begin{equation*}
	\forall m\in \N\, \left(\prod_{i= m}^{\infty} (1-a_i)=0\right).
\end{equation*}
 Hence, it makes sense to also consider a corresponding quantitative hypothesis:
\begin{equation}\tag{$Q^{\star}$}\label{q2'}
	\begin{gathered}
		{\rm A'}:\N\times (0,+\infty) \to\N \text{ is a function satisfying}\\
		\forall \eps >0 \, \forall m\in \N\, \left( \prod_{i=m}^{{\rm A}'(m,\eps)}(1-a_i)\leq \eps\right),
	\end{gathered}
\end{equation}
i.e.\ ${\rm A}'(m, \cdot)$ is a rate of convergence towards zero for the sequence $\left(\prod_{i=m}^{n}(1-a_i)\right)_n$. 

Next we state a quantitative version of Lemma~\ref{L:Xu} which relies on the condition \eqref{q2'} (see  also \cite[Lemma 2.4]{K(15)} and \cite{LLPP(21)}).
\begin{lemma}\label{L:xu_seq_reals_qt2}
	Let $(s_n)$ be a bounded sequence of non-negative real numbers and $D\in\N\setminus\{0\}$ an upper bound on $(s_n)$. Consider sequences of real numbers $(a_n)\subset\, [0,1]$, $(r_n)\subset \R$ and $(v_n)\subset \R^+_0$, and functions ${\rm A'}: \N\times (0, +\infty) \to \N$ and ${\rm R}$, ${\rm V}:(0,+\infty) \to \N$ such that
	\begin{enumerate}[$(i)$]
		\item ${\rm A'}$ satisfies \eqref{q2'},
		\item ${\rm R}$ is such that $\forall \eps>0 \, \forall n\geq {\rm R}(\eps) \, \left( r_n \leq \eps \right)$,
		\item ${\rm V}$ is a Cauchy rate for $(\sum v_n)$.
	\end{enumerate}
	Assume that for all $n\in \N$, $s_{n+1}\leq (1-a_n)s_n+a_nr_n + v_n$. Then	$ \lim s_n=0$ with rate of convergence
	\[
	\theta'[{\rm A'}, {\rm R}, {\rm V}, D](\eps):={\rm A'}\left(K, \frac{\eps}{3D}\right)+1,\, \text{ with }\, K\, \text{ as in Lemma~\ref{L:xu_seq_reals_qt1}}.
	\]
	Moreover,
	\begin{enumerate}[$(1)$]
		\item If $r_n\equiv 0$, then the function $\theta'$ is simplified to
		\[
		\theta'(\eps):=\widehat{\theta}'[{\rm A},{\rm V},D](\eps):={\rm A'}\left({\rm V}\left(\frac{\eps}{2}\right)+1, \frac{\eps}{2D}\right)+1.
		\]
		\item If $v_n\equiv 0$, then the function $\theta'$ is simplified to
		\[
		\theta'(\eps):=\widecheck{\theta}'[{\rm A},{\rm R},D](\eps):={\rm A'}\left({\rm R}\left(\frac{\eps}{2}\right), \frac{\eps}{3D}\right)+1.
		\]
	\end{enumerate}
\end{lemma}
We will also require particular instances of Lemma~\ref{L:Xu} that allow for a error term. The next two results are trivial variants of \cite[Lemmas 5.2 and 5.3]{KL(12)}. For completeness, we include the proofs nevertheless.
\begin{lemma}\label{L:xu_seq_reals_qt3}
	Let $(s_n)$ be a bounded sequence of non-negative real numbers and $D\in\N\setminus\{0\}$ an upper bound on $(s_n)$. Consider sequences of real numbers $(a_n)\subset\, [0,1]$, $(r_n)\subset \R$ and assume that $\sum a_n=\infty$ with a rate of divergence ${\rm A}$. Let $\eps>0$, $K, P\in \N$ be given. If for all $n\in [K,P]$
	\begin{equation*}
		(i)\quad s_{n+1}\leq (1-a_n)s_n+a_nr_n+\mathcal{E}\,,\qquad\qquad (ii)\quad r_n\leq \frac{\eps}{3}\,,\qquad\qquad (iii)\quad \mathcal{E}\leq \frac{\eps}{3(P+1)}, 
	\end{equation*}
	then $\forall n \in [\sigma, P] \left( s_n\leq \eps\right)$, where 
	\begin{equation*}
		\sigma:=\sigma[{\rm A},D](\eps,K):={\rm A}\left(K+\left\lceil\ln\left( \frac{3D}{\eps}\right)\right\rceil\right)+1.
	\end{equation*}
\end{lemma}
\begin{proof}
	First note that if $\sigma>P$, then the result is trivial. Hence we may assume that $\sigma\leq P$ (which in particular implies that $K\leq P$). Next, we may assume that $\eps <D$, otherwise the result also trivially holds. Now, by induction one easily sees that for all $m\leq P-K$,
	\begin{equation*}
		s_{K+m+1}\leq \left( \prod_{i=K}^{K+m} (1-a_i) \right)s_K+\left( 1-\prod_{i=K}^{K+m}(1-a_i) \right)\frac{\eps}{3}+(m+1)\mathcal{E}.
	\end{equation*}
	Hence, for all $m\leq P-K$,
	\begin{equation}\label{e:Xuquant}
		s_{K+m+1}\leq \left( \prod_{i=K}^{K+m} (1-a_i) \right)s_K+\frac{\eps}{3}+(P+1)\mathcal{E}\leq \left( \prod_{i=K}^{K+m} (1-a_i) \right)D+\frac{2\eps}{3}. 
	\end{equation}
	Since $j+1\leq \sum_{i=0}^{{\rm A}(j+1)}a_i\leq {\rm A}(j+1)+1$, we have ${\rm A}(j+1)\geq j$ for all $j\in\N$. By the assumption on $\eps$, we have $3D/\eps\geq 3$ and so $\lceil\ln(3D/\eps)\rceil\geq 1$. Hence ${\rm A}(K+\lceil\ln(3D/\eps)\rceil)\geq K$, and can consider the natural number $M:={\rm A}(K+\lceil\ln(3D/\eps)\rceil)-K$.
	
	For all $m\geq M$, we have
	\begin{equation*}
		\sum_{i=0}^{K+m}a_i\geq \sum_{i=0}^{K+M}a_i=\sum_{i=0}^{{\rm A}(K+\lceil\ln(3D/\eps)\rceil)}a_i\geq K+\ln\left(\frac{3D}{\eps}\right)\geq \sum_{i=0}^{K-1}a_i+ \ln\left(\frac{3D}{\eps}\right).
	\end{equation*}
	which entails $\sum_{i=K}^{K+m}a_i\geq \ln(3D/\eps)$. Using the inequality $\forall x\in\R^+_0 \left(1-x\leq \exp(-x)\right)$, we derive for all $m\geq M$,
	\begin{equation*}
		\left(\prod_{i=K}^{K+m}(1-a_i)\right)D\leq \exp\left(-\sum_{i=K}^{K+m}a_i\right)D\leq \frac{\eps}{3D}D=\frac{\eps}{3}. 
	\end{equation*}
	Together with \eqref{e:Xuquant}, we have thus concluded that for $m\in[M, P-K]$,
	\begin{equation*}
		s_{K+m+1}\leq \left( \prod_{i=K}^{K+m} (1-a_i) \right)D+\frac{2\eps}{3}\leq \eps,
	\end{equation*}
	which entails the result.
\end{proof}
As before, instead of a rate of divergence ${\rm A}$, we can stablish an analogous of the previous result with a function ${\rm A'}$ satisfying \eqref{q2'}.
\begin{lemma}\label{L:xu_seq_reals_qt4}
	Let $(s_n)$ be a bounded sequence of non-negative real numbers and $D\in\N\setminus\{0\}$ an upper bound on $(s_n)$. Consider sequences of real numbers $(a_n)\subset\, [0,1]$, $(r_n)\subset \R$ assume that $\sum a_n=\infty$ with a function ${\rm A'}$ satisfying \eqref{q2'}. Let $\eps>0$, $K, P\in \N$ be given. If for all $n\in [K,P]$
	\begin{equation*}
		(i)\quad s_{n+1}\leq (1-a_n)s_n+a_nr_n+\mathcal{E}\,,\qquad\qquad (ii)\quad r_n\leq \frac{\eps}{3}\,,\qquad\qquad (iii)\quad \mathcal{E}\leq \frac{\eps}{3(P+1)}, 
	\end{equation*}
	then $\forall n \in [\sigma', P] \left( s_n\leq \eps\right)$, where 
	\begin{equation*}
		\sigma':=\sigma'[{\rm A'},D](\eps,K):={\rm A'}\left(K, \frac{\eps}{3D}\right)+1.
	\end{equation*}
\end{lemma}
\begin{proof}
	Following the proof of Lemma~\ref{L:xu_seq_reals_qt3}, we have for all $m\leq P-K$,
	\begin{equation*}
	s_{K+m+1}\leq \left( \prod_{i=K}^{K+m} (1-a_i) \right)D+\frac{2\eps}{3}.
	\end{equation*}
	Again from the assumption on $\eps$, we may consider the natural number $M:={\rm A'}\left(K, \eps/3D\right)-K$. For all $m\geq M$,
	\begin{equation*}
		\left(\prod_{i=K}^{K+m}(1-a_i)\right)D\leq \left(\prod_{i=K}^{K+M}(1-a_i)\right)D=\left(\prod_{i=K}^{{\rm A'}\left(K, \eps/3D\right)}(1-a_i)\right)D\leq \frac{\eps}{3},
	\end{equation*}
	which as before entails the result.
\end{proof}
\section{Asymptotic regularity}\label{s:asymptoticregularity}
In this section we prove that a sequence $(x_n)$ generated by \eqref{e:MannHalpern} is an asymptotically regular sequence of approximate common fixed points of $T$ and $U$.
We will assume that the conditions $(i)-(v)$ of Theorem~\ref{t:main} hold and that their corresponding quantitative information is given explicitly. Namely, given $(\alpha_n),(\beta_n)\subseteq [0,1]$, we consider the following conditions: 
\begin{enumerate}[($Q_1$)]
	\item\label{Q1} $\Gamma_1:(0,+\infty)\to \N$ is a rate of convergence for $\lim \alpha_n=0$, i.e.
	\[
	\forall \eps >0 \, \forall n \geq \Gamma_1(\eps) \left(\alpha_n \leq \eps \right).
	\]
	\item\label{Q2} $\Gamma_2:\N\to\N$ is a rate of divergence for $\left( \sum_{n\geq 0} \alpha_n\right)$, i.e.
	\[
	\forall k \in \N \left(\sum_{i=0}^{\Gamma_2(k)}\alpha_i \geq k \right).
	\]
	\item\label{Q3}  $\Gamma_3:(0,+\infty)\to \N$ is a Cauchy rate  for $\left(\sum_{n\geq 0}|\alpha_{n+1}-\alpha_{n}|\right)$, i.e. 
	\[\forall \eps >0 \, \forall n \in \N \left(\sum_{i=\Gamma_3(\eps)+1}^{\Gamma_3(\eps)+n}|\alpha_{i+1}-\alpha_{i}|\leq \eps \right).\]
	\item\label{Q4}  $\Gamma_4:(0,+\infty)\to \N$ is a Cauchy rate  for $\left(\sum_{n\geq 0}|\beta_{n+1}-\beta_{n}|\right)$.
	\item\label{Q5} $\gamma\in(0,1/2]$ is such that $\gamma \leq \beta_n \leq 1-\gamma$, for all $n \in \N$.
\end{enumerate}
We will assume furthermore that the functions above satisfy the natural monotonicity conditions:
\begin{enumerate}
	\item if $\eps \leq \eps' $ then $\Gamma_i(\eps)\geq \Gamma_i(\eps')$, for $i=1,3,4$
	\item if $k \leq k' $ then $\Gamma_2(k)\leq \Gamma_2(k')$.
\end{enumerate}
\begin{example}\label{ex:1}
	Consider the sequences defined by $\alpha_n:=\frac{1}{n+1}$, $\beta_n:=\beta \in (0,1)$ for all $n \in \N$. Then the conditions $(Q_{\ref{Q1}})-(Q_{\ref{Q5}})$ are satisfied with 
	\[
	\Gamma_1(\eps)=\Gamma_3(\eps):= \lfloor \eps^{-1}\rfloor, \quad \Gamma_2(k):=\lfloor e^{k}\rfloor, \quad \Gamma_4 :\equiv 0, \quad \text{ and } \quad \gamma:=\min \{ \beta, 1- \beta\}.
	\]
\end{example}
We start by showing that $(x_n)$ is bounded.
\begin{lemma}\label{l:bounded}
	Consider a natural number $N \in \N \setminus \{0\}$ such that $N \geq \max\{d(x_0,p),d(u,p)\}$, for some $p \in F$.
	We have that \begin{equation}\label{e:bounded}
		d(x_{n},p) \leq N
	\end{equation}
	and, in particular, the sequence $(x_n)$ is bounded.
\end{lemma}
\begin{proof}
	Since  $U$ is nonexpansive and $U(p)=p$, using (W1) we have
	\begin{equation}\label{e:evenindices}
		d(x_{2n+2},p)\leq (1-\alpha_n)d(U(x_{2n+1}),p)+\alpha_nd(x_{2n+1},p)= d(x_{2n+1},p).
	\end{equation}
	By induction, we have that $\forall n \in \N \left(d(x_{2n},p)\leq N \right)$.
	The base case is trivial. If the induction hypothesis holds for $n$, using the fact that $T$ is nonexpansive, $T(p)=p$, and again (W1), we get
	\begin{equation}\label{e:oddindices}
		d(x_{2n+1},p)\leq (1-\alpha_n)d(x_{2n},p)+\alpha_n d(u,p) \leq (1-\alpha_n)N+\alpha_n N=N.
	\end{equation}
	By \eqref{e:evenindices}, the induction step follows. Moreover, from \eqref{e:oddindices} it is clear that the bound is also valid for the odd terms, and we conclude the result.
\end{proof}
The next lemma allows to prove the asymptotically regularity of the sequence $(x_n)$.
\begin{lemma}\label{l:artheta}
	Consider a natural number $N \in \N \setminus \{0\}$ such that $N \geq \max\{d(x_0,p),d(u,p)\}$, for some $p \in F$.
	If conditions $(Q_{\ref{Q1}})-(Q_{\ref{Q5}})$ hold, then
	\begin{enumerate}
		\item[$(i)$] $\theta_1$ is a rate of convergence for $\lim d(x_{2n+2},x_{2n})=0$,
		\item[$(ii)$] $\theta_2$ is a rate of convergence for $\lim d(x_{2n+3},x_{2n+1})=0$,
		\item[$(iii)$] $\theta_3$ is a rate of convergence for both $\lim d(U(x_{2n+1}),x_{2n+1})=0$ and for $\lim d(x_{2n+2}, x_{2n+1})=0$,
		\item[$(iv)$] $\theta_4$ is a rate of convergence for $\lim d(x_{2n+1},x_{2n})=0$.
	\end{enumerate}
	where
	\begin{equation*}
		\begin{aligned}
			\theta_1(\eps)&:=\widehat{\theta}[{\rm A},{\rm V},D], \, \text{ and } \widehat{\theta} \text{ is as in Lemma~\ref{L:xu_seq_reals_qt1}(1), with parameters }\\
			&\bullet \, {\rm A}(k):=\Gamma_2(k+1), \, \text{ for all }\, k\in \N\\
			&\bullet\, {\rm V}(\eps):=\max\left\{\Gamma_3\left(\frac{\eps}{4N}\right),\Gamma_4\left(\frac{\eps}{4N}\right)\right\},\, \text{ for all } \, \eps>0\\
			&\bullet \, D:=2N - \text{which by \eqref{e:bounded} is a bound on the sequence $(d(x_{2n+2}, x_{2n}))$}\, \\
			\theta_2(\eps)&:=\max\left\{\theta_1\left(\frac{\eps}{2}\right), \Gamma_3\left(\frac{\eps}{4N}\right)+1\right\}\\
			\theta_3(\eps)&:=\max\left\{ \theta_1\left(\frac{(\gamma\eps)^2}{8N}\right), \Gamma_1\left(\frac{(\gamma\eps)^2}{2N^2}\right) \right\}\\
			\theta_4(\eps)&:=\max\left\{ \theta_1\left(\frac{\eps}{2}\right), \theta_3\left(\frac{\eps}{2}\right) \right\}.
		\end{aligned}
	\end{equation*}
\end{lemma}
\begin{proof}
	Using (W2) and (W4) we have
	\SE{d(x_{2n+3},x_{2n+1})& \leq d(x_{2n+3},(1-\alpha_{n+1})T(x_{2n})\oplus\alpha_{n+1}u) +d((1-\alpha_{n+1})T(x_{2n})\oplus\alpha_{n+1}u,x_{2n+1})\\
		& \leq (1-\alpha_{n+1})d(T(x_{2n+2}),T(x_{2n}))+|\alpha_{n+1}-\alpha_{n}|d(T(x_{2n}),u)\\
		& \leq(1-\alpha_{n+1})d(x_{2n+2},x_{2n})+|\alpha_{n+1}-\alpha_{n}|d(T(x_{2n}),u).
	}
	Hence, using \eqref{e:bounded}, for all $n \in \N$
	\begin{equation}\label{e:difp}
		d(x_{2n+3},x_{2n+1})\leq (1-\alpha_{n+1}) d(x_{2n+2},x_{2n})+2N|\alpha_{n+1}-\alpha_{n}|.
	\end{equation}
	Similarly,
	\SE{d(x_{2n+4},x_{2n+2}) & \leq d(x_{2n+2},(1-\beta_{n+1})U(x_{2n+1}) \oplus \beta_{n+1}x_{2n+1})\\
		& \quad +d((1-\beta_{n+1})U(x_{2n+1}) \oplus \beta_{n+1}x_{2n+1},x_{2n+2})\\
		& \leq  (1-\beta_{n+1})d(U(x_{2n+3}),U(x_{2n+1}))+\beta_{n+1}d(x_{2n+3},x_{2n+1})\\
		& \quad +|\beta_{n+1}-\beta_{n}|d(U(x_{2n+1}),x_{2n+1})\\
		& \leq d(x_{2n+3},x_{2n+1})+|\beta_{n+1}-\beta_{n}|d(U(x_{2n+1}),x_{2n+1}).
	}
	Hence, using \eqref{e:bounded}, for all $n \in \N$
	\begin{equation}\label{e:difi}
		d(x_{2n+4},x_{2n+2})\leq d(x_{2n+3},x_{2n+1})+2N|\beta_{n+1}-\beta_{n}|.
	\end{equation}
	From \eqref{e:difp} and \eqref{e:difi} we conclude that for all $n \in \N$
	\SE{d(x_{2n+4},x_{2n+2})& \leq (1-\alpha_{n+1}) d(x_{2n+2},x_{2n})+2N\left(|\alpha_{n+1}-\alpha_{n}|+|\beta_{n+1}-\beta_{n}|\right).}
	
	By $(Q_{\ref{Q2}}),(Q_{\ref{Q3}})$ and $(Q_{\ref{Q4}})$, we may apply Lemma~\ref{L:xu_seq_reals_qt1} with $s_n=d(x_{2n+2},x_{2n})$, $a_n=\alpha_{n+1}$, $r_n\equiv 0$, and $v_n=2N\left(|\alpha_{n+1}-\alpha_{n}|+|\beta_{n+1}-\beta_{n}|\right)$  to conclude that 
	\begin{equation}\label{e:difpares}
		\lim d(x_{2n+2},x_{2n})=0, \text{ with rate of convergence $\theta_1$}.
	\end{equation}
	Indeed, for all $k\in \N$,
	\begin{equation*}
		\sum_{i=0}^{{\rm A}(k)}a_i=\sum_{i=0}^{\Gamma_2(k+1)}\alpha_{i+1}\geq \sum_{i=0}^{\Gamma_2(k+1)}\alpha_{i}-\alpha_0 \geq k
	\end{equation*}
	and, for all $\eps >0$ and $n \in\N$,
	\SE{
		\sum_{i={\rm V}(\eps)+1}^{{\rm V}(\eps)+n}v_i&=2N\left(\sum_{i={\rm V}(\eps)+1}^{{\rm V}(\eps)+n} |\alpha_{i+1}-\alpha_{i}|+\sum_{i={\rm V}(\eps)+1}^{{\rm V}(\eps)+n} |\beta_{i+1}-\beta_i|\right)\\[2mm]
		&=2N\left(\sum_{i=\Gamma_3(\frac{\eps}{4N})+j+1}^{\Gamma_3(\frac{\eps}{4N})+j+n} |\alpha_{i+1}-\alpha_{i}|+\sum_{i=\Gamma_4(\frac{\eps}{4N})+j'+1}^{\Gamma_4(\frac{\eps}{4N})+j'+n} |\beta_{i+1}-\beta_i|\right)\\[2mm]
		&\leq 2N\left(\sum_{i=\Gamma_3(\frac{\eps}{4N})+1}^{\Gamma_3(\frac{\eps}{4N})+j+n} |\alpha_{i+1}-\alpha_{i}|+\sum_{i=\Gamma_4(\frac{\eps}{4N})+1}^{\Gamma_4(\frac{\eps}{4N})+j'+n} |\beta_{i+1}-\beta_i|\right)\\
		&\leq 2N \left(\frac{\eps}{4N}+\frac{\eps}{4N}\right)= \eps.
	}
	From $(Q_{\ref{Q3}})$, for any $\eps >0$ and $n\geq \Gamma_3(\eps/4N)+1$, we have $|\alpha_{n+1}-\alpha_n|\leq \eps/4N$. Then, by \eqref{e:difp}, 
	\begin{equation}\label{e:difimpares}
		\lim d(x_{2n+3},x_{2n+1})=0, \, \text{ with rate of convergence }\, \theta_2.
	\end{equation}
	With $p\in F$ as in the hypothesis over $N$, using \eqref{CN} and the fact that $U$ is nonexpansive, we have
	\SE{d^2(x_{2n+2},p)& \leq (1-\beta_{n})d^2(U(x_{2n+1}),p)+\beta_{n}d^2(x_{2n+1},p)-\beta_n(1-\beta_n)d^2(U(x_{2n+1}),x_{2n+1})\\
		& \leq d^2(x_{2n+1},p)-\beta_n(1-\beta_n)d^2(U(x_{2n+1}),x_{2n+1})
	}
	and similarly, now using the fact that $T$ is nonexpansive,
	\SE{d^2(x_{2n+1},p)& \leq (1-\alpha_{n})d^2(T(x_{2n}),p)+\alpha_{n}d^2(u,p)-\alpha_{n}(1-\alpha_{n})d^2(T(x_{2n}),u)\\
		& \leq d^2(x_{2n},p)+\alpha_{n}d^2(u,p).
	}
	Then
	\SE{d^2(x_{2n+2},p)&\leq d^2(x_{2n},p)+\alpha_{n}d^2(u,p)-\beta_{n}(1-\beta_{n})d^2(U(x_{2n+1}),x_{2n+1}),
	}
	which implies
	\SE{\beta_{n}(1-\beta_{n})d^2(U(x_{2n+1}),x_{2n+1})&\leq  d^2(x_{2n},p)-d^2(x_{2n+2},p)+\alpha_{n}d^2(u,p)\\
		& \leq d(x_{2n+2},x_{2n})(d(x_{2n+2},x_{2n})+2d(x_{2n+2},p))+\alpha_{n}d^2(u,p).
	}
	Now, condition $(Q_{\ref{Q5}})$ entails
	\begin{equation}\label{e:ineqforUimpar}
		\left(\gamma \cdot d(U(x_{2n+1}),x_{2n+1})\right)^2\leq d(x_{2n+2},x_{2n})(d(x_{2n+2},x_{2n})+2d(x_{2n+2},p))+\alpha_{n}d^2(u,p).
	\end{equation}
	For $\eps >0$, using \eqref{e:bounded} and \eqref{e:difpares}, we conclude that for $n\geq \theta_1((\gamma\eps)^2/8N)$,
	\begin{equation*}
		d(x_{2n+2},x_{2n})(d(x_{2n+2},x_{2n})+2d(x_{2n+2},p))\leq d(x_{2n+2},x_{2n})(2N+2N)\leq \frac{(\gamma\eps)^2}{2}.
	\end{equation*}  
	Moreover, from condition $(Q_{\ref{Q1}})$, we get for $n\geq \Gamma_1((\gamma\eps)^2/2N^2)$
	\begin{equation*}
		\alpha_nd^2(u,p)\leq \alpha_nN^2\leq \frac{(\gamma\eps)^2}{2}.
	\end{equation*}
	Hence, from \eqref{e:ineqforUimpar} we conclude that
	\begin{equation}\label{e:Uimpar}
		\lim d(U(x_{2n+1}),x_{2n+1})=0, \, \text{ with rate of convergence }\, \theta_3.
	\end{equation}
	Since $d(x_{2n+2},x_{2n+1})= (1-\beta_n)d(U(x_{2n+1}),x_{2n+1})$ by \eqref{e:equality}, from \eqref{e:Uimpar} we also get that $\lim d(x_{2n+2},x_{2n+1})=0$ with rate of convergence $\theta_3$. On the other hand, since $d(x_{2n+1},x_{2n})\leq d(x_{2n+2},x_{2n+1})+d(x_{2n+2},x_{2n})$ it follows that $\lim d(x_{2n+1},x_{2n})=0$ with rate of convergence $\theta_4$.
\end{proof}
\begin{example}
	With $(\alpha_n), (\beta_n)$ as in Example~\ref{ex:1}, it is possible to verify that we have the following rates of convergence
	\begin{enumerate}
		\item[] $\eps \mapsto \left\lfloor \exp\left(\frac{12N}{\eps}+2\right)\right\rfloor +1$  for $\lim d(x_{2n+2}, x_{2n})=0$,
		\item[] $\eps \mapsto \left\lfloor \exp\left(\frac{24N}{\eps}+2\right)\right\rfloor +1$ for $\lim d(x_{2n+3}, x_{2n+1})=0$,
		\item[] $\eps \mapsto \left\lfloor \exp\left(\left(\frac{14N}{\gamma\eps}\right)^2+2\right)\right\rfloor +1$ for both $\lim d(U(x_{2n+1}), x_{2n+1})=0 \text{ and }  \lim d(x_{2n+2},x_{2n+1})=0,$
		\item[] $\eps \mapsto \left\lfloor \exp\left(\left(\frac{20N}{\gamma\eps}\right)^2+2\right)\right\rfloor +1$ for  $\lim d(x_{2n+1},x_{2n})=0$.
	\end{enumerate}
\end{example}

\begin{remark}
	Lemma~\ref{l:artheta}, and Proposition~\ref{p:asymptoticregularity} below, still hold if instead of a function $\Gamma_2$ satisfying $(Q_{\ref{Q2}})$, we have a monotone\footnote{I.e.\ $\Gamma'_2$ satisfies $\forall \eps_1, \eps_2 >0 \, \forall m_1, m_2 \in \N\, \left( \eps_1 \leq \eps_2 \land m_1\leq m_2 \to {\Gamma'_2}(m_1, \eps_2)\leq {\Gamma'_2}(m_2, \eps_1)\right).$} function $\Gamma'_2:\N\times(0,+\infty)\to \N$ such that for all $m\in \N$, $\Gamma'_2(m, \cdot)$ is a rate of convergence towards zero for the sequence $(\prod_{i=m}^{n}(1-\alpha_i))$. In this case, one defines $\theta_1$ using Lemma~\ref{L:xu_seq_reals_qt2}(1) instead of Lemma~\ref{L:xu_seq_reals_qt1}(1), namely $\theta_1=\widehat{\theta}'[{\rm A'}, {\rm V}, 2N]$, where ${\rm A'}(m,\eps):=\Gamma_2'(m+1,\eps)$, and the remaining functions are the same as before.
\end{remark}
We now compute rates of convergence for the asymptotic properties of the sequence $(x_n)$ using Lemma~\ref{l:artheta}.
\begin{proposition}\label{p:asymptoticregularity}
	Under the conditions of Lemma~\ref{l:artheta}, we have that
	\[
	\lim d(x_{n+1},x_n)=\lim d(U(x_n),x_n)=\lim d(T(x_n), x_n)=0,
	\]
	with monotone rates of convergence, respectively,
	\begin{enumerate}
		\item[$(i)$] $\rho_1(\eps):=\max\{2\theta_3(\eps)+1, 2\theta_4(\eps)\}$
		\item[$(ii)$] $\rho_2(\eps):=2\theta_3\left(\frac{\eps}{3}\right)+2$
		\item[$(iii)$] $\rho_3(\eps):=2\max\left\{ \theta_4\left(\frac{\eps}{6}\right), \Gamma_1(\frac{\eps}{4N}) \right\}+1$,
	\end{enumerate}
	where $\theta_1$, $\theta_2$, $\theta_3$ and $\theta_4$ are as in Lemma~\ref{l:artheta}.
\end{proposition}
\begin{proof}
	Let $\eps>0$ be given. For part $(i)$, take $n\geq \rho_1(\eps)$. If $n=2n'+1$, then $n'\geq \theta_3(\eps)$ and the result follows from Lemma~\ref{l:artheta}$(iii)$. Similarly, if $n=2n'$ we get $n'\geq \theta_4(\eps)$ and the result now follows from Lemma~\ref{l:artheta}$(iv)$. 
	For part $(ii)$, using the fact that $U$ is nonexpansive we have
	\SE{d(U(x_{2n+2}),x_{2n+2})&\leq d(U(x_{2n+2}),U(x_{2n+1}))+d(U(x_{2n+1}),x_{2n+1})+d(x_{2n+2},x_{2n+1})\\
		&\leq 2d(x_{2n+2},x_{2n+1})+d(U(x_{2n+1}),x_{2n+1})\leq 3d(U(x_{2n+1}),x_{2n+1}),
	}
	which from Lemma~\ref{l:artheta}$(iii)$ implies that $\lim d(U(x_{2n+2}),x_{2n+2})=0$ with rate of convergence $\eps\mapsto\theta_3(\eps/3)$, and consequently 
	\begin{equation*}
		\lim d(U(x_n),x_n)=0, \, \text{ with rate of convergence } \tilde{\rho}_2(\eps):=\max\left\{ 2\theta_3(\eps)+1, 2\theta_3\left(\frac{\eps}{3}\right)+2 \right\}.
	\end{equation*}
	Note that the assumption of monotonicity on the functions $\Gamma_i$ (and the definition of the function $\theta$ in Lemma~\ref{L:xu_seq_reals_qt1}), entail that the functions $\theta_i$ (and $\rho_i$) are also monotone. In particular, the monotonicity of $\theta_3$ entails that $\theta_3(\eps/3)\geq \theta_3(\eps)$, and so $\tilde{\rho}_2$ coincides with $\rho_2$.

	Finally for part $(iii)$, using (W1), the fact that $T$ is a nonexpansive map and \eqref{e:bounded}, we obtain
	\SE{d(T(x_{2n+1}),x_{2n+1})&\leq (1-\alpha_n)d(T(x_{2n+1}), T(x_{2n})) + \alpha_n d(T(x_{2n+1}), u)\\
		&\leq d(x_{2n+1},x_{2n})+2N\alpha_n
	}
	and
	\SE{d(T(x_{2n}), x_{2n})&\leq d(T(x_{2n+1}),T(x_{2n}))+d(T(x_{2n+1}), x_{2n+1}) + d(x_{2n+1},x_{2n})\\
		&\leq 2d(x_{2n+1},x_{2n})+d(T(x_{2n+1}),x_{2n+1})\\
		&\leq 3d(x_{2n+1},x_{2n}) + 2N\alpha_n.
	}
	By $(Q_{\ref{Q1}})$ and Lemma~\ref{l:artheta}$(iv)$, we conclude that
	\begin{equation*}
		\begin{aligned}
			&\eps\mapsto \max\left\{ \theta_4\left(\frac{\eps}{2}\right), \Gamma_1\left(\frac{\eps}{4N}\right) \right\} \, \text{ is a rate of convergence for }\, \lim d(T(x_{2n+1}), x_{2n+1})=0\\
			&\eps\mapsto \max\left\{ \theta_4\left(\frac{\eps}{6}\right), \Gamma_1\left(\frac{\eps}{4N}\right) \right\}\, \text{ is a rate of convergence for }\, \lim d(T(x_{2n}), x_{2n})=0.
		\end{aligned}
	\end{equation*}
	Using the monotonicity of $\theta_4$, we conclude part $(iii)$.
\end{proof}
\begin{example}
	With $(\alpha_n), (\beta_n)$ as in Example~\ref{ex:1}, it is possible to verify that for any $\eps>0$,
	\begin{enumerate}
		\item[] $n\geq \left\lfloor \exp\left(\left(\frac{20N}{\gamma\eps}\right)^2+3\right)\right\rfloor +2 \, \to\,  d(x_{n+1},x_{n})\leq \eps$,
		\item[] $n\geq \left\lfloor \exp\left(\left(\frac{14N}{\gamma\eps}\right)^2+3\right)\right\rfloor +3 \, \to\,  d(U(x_{n}),x_{n})\leq \eps$,
		\item[] $n\geq \left\lfloor \exp\left(\left(\frac{20N}{\gamma\eps}\right)^2+3\right)\right\rfloor +3 \, \to\,  d(T(x_{n}),x_{n})\leq \eps$.
	\end{enumerate}
\end{example}
We have the following trivial consequence of Proposition~\ref{p:asymptoticregularity}.
\begin{lemma}\label{l:rho}
For any $\eps>0$ define $\rho(\eps):=\max\{\rho_2(\eps),\rho_3(\eps) \}$, where $\rho_2,\rho_3$ are as in Proposition~\ref{p:asymptoticregularity}. Then $\rho$ is a monotone rate of asymptotic regularity for both $U$ and $T$, i.e.\
\begin{equation*}
\forall \eps >0 \, \forall n \geq \rho(\eps) \left(d(x_n,T(x_n), d(x_n,U(x_n))\leq \eps) \right). 
\end{equation*}
\end{lemma}
\section{Metric Projection in {\rm CAT(0)} spaces}\label{s:Projection}

Let $S$ be a nonempty closed convex subset of a complete $\CAT$ space $X$. The \emph{metric projection} onto $S$ is the mapping $P_S: X \to S$ defined, for each $u \in X$,  as the unique point  $P_Su$ satisfying  $d(u,P_Su)= \min_{x\in S} d(u,x)$, \cite{BH(13)}. In this section, we show a useful quantitative version regarding the metric projection. The proof follows the arguments from \cite{FFLLPP(19),PP(ta)} (see also \cite{kohlenbach2011quantitative}). We point out that our analysis only requires an $\eps$-weakening of the metric projection (similarly to \cite{kohlenbach2011quantitative}), and thus the results in this section are valid even without assuming the completeness of the space.

Let $T,U:C \to C$ be nonexpansive maps on a nonempty and convex subset of a $\CAT$ space $X$ and assume that $F:=\Fix(T) \cap \Fix(U) \neq \emptyset$.  
Given $\eta >0$ and $a \in X$ we write $$F_N(a,\eta):=\{x \in C: d(x,T(x)),d(x,U(x))\leq \eta\} \cap \overline{B}_N(a).$$
\begin{proposition}\label{P:metric}
Given $u \in C$, let $N \in \N \setminus \{0\}$ be such that $N \geq 2d(u,p)$, for some $p \in F$. Then, for every $\eps >0$ and $\delta:(0,1] \to (0,1]$ there exists $\eta \in [\varphi(\eps,\delta),1]$ and $x \in F_N(p,\delta(\eta))$ such that
\begin{equation*}
\forall y \in F_N(p,\eta) \left(d^2(u,x)\leq d^2(u,y)+\eps \right),
\end{equation*}
where $\varphi(\eps,\delta):=\varphi[N](\eps,\delta):= \min \{\delta^{(i)}(1): i \leq r\}$ with $r:=r(N,\eps):=\lceil \frac{N^2}{4\eps}\rceil$.\\
Moreover, $\eta$ is witnessed by $\delta^{(i)}(1)$, for some $i \leq r$.
\end{proposition}
\begin{proof}
Suppose that for some $\eps >0$ and $\delta:(0,1] \to (0,1]$ we have
\begin{equation}\label{e:contradiction}
\forall \eta \in [\varphi(\eps,\delta),1] \, \forall x \in F_{N}(p,\delta(\eta))\, \exists y \in F_N(p,\eta) \left( d^2(u,y)< d^2(u,x) -\eps \right).
\end{equation}
We construct a finite sequence $x_0, \dots, x_{r}$ of elements of $C$ as follows. We define $x_0:=p$. In particular, $x_0 \in F_N(p,\delta^{(r)}(1))$.
Assume that, for $j \leq r-1$ we have $x_j \in F_N(p, \delta^{(r-j)}(1))$. Then, by \eqref{e:contradiction} there exists $y_0 \in F_N(p, \delta^{(r-j-1)}(1))$ such that 
\begin{equation*}
d^2(u,y_0)< d^2(u,x_j) -\eps,
\end{equation*}
and define $x_{j+1}:=y_0$.%

By construction, for all $j \leq r$ we have $d^2(u,x_{j+1})<d^2(u,x_j)-\eps$ and so 
\begin{equation*}
d^2(u,x_{r})<d^2(u,x_0)-r\eps \leq \frac{N^2}{4}- \frac{N^2}{4\eps}\eps=0, 
\end{equation*} 
which is a contradiction.
\end{proof}
The following lemmas are a generalization to $\CAT$ spaces of Lemmas~2.3 and 2.7 of \cite{kohlenbach2011quantitative}. The proofs are similar but we include them nevertheless for completeness.
\begin{lemma}\label{L:2.3K}
Let $X$ be a $\CAT$ space, $C \subseteq X$ a convex bounded subset and $b \in \N \setminus\{0\}$, a bound on the diameter of $C$. Consider $T:C \to C$ a nonexpansive map on $C$. Then 
\begin{equation*}
\forall \eps >0 \, \forall x_1,x_2 \in C \left(\bigwedge_{i=1}^{2} d(x_i,T(x_i))< \frac{\eps^2}{12b} \to \forall t \in [0,1] \left(d(w_t,T(w_t))< \eps \right)\right),
\end{equation*} 
where $w_t:= (1-t)x_1 \oplus tx_2$.
\end{lemma}
\begin{proof}
Let $x_1,x_2 \in C$ be such that 
\begin{equation}\label{e:wedge}
\bigwedge_{i=1}^{2} d(x_i,T(x_i))< \frac{\eps^2}{12b}.
\end{equation}
For any $t \in [0,1]$ we have $K:=\max\{d(x_1,w_t),d(x_1,T(w_t))\}\leq b$. If $K < \eps/2$, then $d(w_t,T(w_t))\leq d(x_1,w_t)+d(x_1,T(w_t))\leq \eps$. So, we may then assume that $K\geq\frac{\eps}{2}$. Since $T$ is nonexpansive, using \eqref{e:wedge} we derive
\SE{d\left(x_2,\frac{1}{2}w_t \oplus \frac{1}{2}T(w_t)\right)& \overset{\mathrm{(W1)}}{\leq} \frac{1}{2} d(x_2,w_t)+\frac{1}{2}d(x_2,T(w_t))\\
& \overset{\phantom{\mathrm{(W1)}}}{\leq} \frac{1}{2} d(x_2,w_t)+ \frac{1}{2} d(T(x_2),T(w_t))+\frac{1}{2} d(x_2,T(x_2))\\
&  \overset{\phantom{\mathrm{(W1)}}}{<} d(x_2,w_t)+\frac{\eps^2}{24b}\\
&  \overset{\phantom{\mathrm{(W1)}}}{\leq} d(x_2,w_t)+\frac{\eps^2}{24K}.
}
We have that 
\begin{equation}\label{e:elsecontradiction}
d\left(x_1,\frac{1}{2}w_t \oplus \frac{1}{2}T(w_t)\right)> d(x_1,w_t)-\frac{\eps^2}{24K},
\end{equation} 
otherwise we would obtain the following contradiction
\SE{d(x_1,x_2)& \leq d\left(x_1,\frac{1}{2}w_t \oplus \frac{1}{2}T(w_t)\right)+d\left(x_2,\frac{1}{2}w_t \oplus \frac{1}{2}T(w_t)\right)\\
&< d(x_1,w_t)-\frac{\eps^2}{24K}+ d(x_2,w_t)+\frac{\eps^2}{24K}\\
& = td(x_1,x_2)+(1-t)d(x_1,x_2)=d(x_1,x_2).
}
Since $d(x_1, T(w_t))< d(x_1,w_t)+\frac{\eps^2}{12K}$, from \eqref{e:elsecontradiction} we obtain 
\SE{d\left(x_1,\frac{1}{2}w_t \oplus \frac{1}{2}T(w_t)\right) > d(x_1,w_t)-\frac{\eps^2}{24K} > K-\frac{\eps^2}{12K}-\frac{\eps^2}{24K}=K\left(1-\frac{(\eps/K)^2}{8}\right)
.}
From Lemma~\ref{l:CATconvexity} with $a=x_1$, $x=w_t$, $y=T(w_t)$, $r=K$ and $\eps=\frac{\eps}{K}\in (0,2]$ we obtain 
\begin{equation*}
d(w_t,T(w_t))< K\frac{\eps}{K}=\eps.\qedhere
\end{equation*}
\end{proof}
\begin{lemma}\label{L:2.7K}
Let $X$ be a $\CAT$ space and $x,y,u \in  X, t \in [0, 1]$ and define
$w_t:= (1-t)x \oplus ty$. Then 
\begin{equation*}
\forall \eps \in (0,b^2] \left(\forall t \in [0,1] \left(d^2(u,x) \leq d^2(u,w_{t})+\frac{\eps^2}{b^2} \right)\to \langle \overrightarrow{ux},\overrightarrow{yx}\rangle \leq \eps \right),
\end{equation*}
where $b \geq d(x,y)$.
\end{lemma}
\begin{proof}
We have, for any $t \in [0,1]$
\SE{d^2(u,w_t)&\leq (1-t)d^2(u,x)+td^2(u,y)-t(1-t)d^2(x,y)\\ 
& =d^2(u,x)+ 2t\left(\frac{1}{2} \left(d^2(u,y)+d^2(x,x)-d^2(u,x)-d^2(x,y) \right) \right)+t^2d^2(x,y)\\
&=d^2(u,x)+2t \langle \overrightarrow{ux},\overrightarrow{xy}\rangle+t^2d^2(x,y).
}
Assume that $d^2(u,x) \leq d^2(u,w_{t})+\frac{\eps^2}{b^2}$. Hence, for $t=\frac{\eps}{b^2} \in [0,1]$ we have
\SE{-\frac{\eps^2}{b^2}\leq 2\frac{\eps}{b^2}\langle \overrightarrow{ux},\overrightarrow{xy}\rangle +\left(\frac{\eps}{b^2}\right)^2d^2(x,y)\leq 2\frac{\eps}{b^2}\langle \overrightarrow{ux},\overrightarrow{xy}\rangle +\left(\frac{\eps}{b^2}\right)^2b^2=-2\frac{\eps}{b^2}\langle \overrightarrow{ux},\overrightarrow{yx}\rangle +\frac{\eps^2}{b^2}, }
from which we derive that $\langle \overrightarrow{ux},\overrightarrow{yx}\rangle \leq \eps$.
\end{proof}
By combining the previous results we obtain a quantitative version of the characterization of the metric projection in $\CAT$ spaces (cf. Lemma~\ref{l:charactProj}).
\begin{proposition}\label{p:innerproduct}
Given $u \in C$, let $N \in \N \setminus\{0\}$ be such that $N\geq 2 d(u,p)$ for some $p \in F$. For any
$\eps >0$ and function $\delta : (0, 1] \to (0, 1]$, there exist $\eta \in [\Phi(\eps,\delta),1]$ and $x \in F_N(p,\delta(\eta))$ such that
\begin{equation*}
\forall y \in F_N(p,\eta) \left(\langle \overrightarrow{ux},\overrightarrow{yx}\rangle \leq \eps \right), 
\end{equation*}
where $\Phi(\eps,\delta):=\Phi[N](\eps,\delta):=\frac{\varphi(\tilde{\eps}, \tilde{\delta})^2}{24N}$, with $\tilde{\eps}:=\frac{\eps^2}{4N^2}$, $\tilde{\delta}(\xi):=\min \left\{\delta\left(\frac{\xi^2}{24N}\right),\frac{\xi^2}{24N}\right\}$, for any $\xi \in (0,1]$, and $\varphi:=\varphi[N]$ is as in Proposition~\ref{P:metric}.
\end{proposition}
\begin{proof}
Let $\eps >0$ and $\delta : (0, 1] \to (0, 1]$ be given. We may assume that $\eps \leq 4N^2$, otherwise the result follows easily using the Cauchy-Schwarz inequality. By Proposition~\ref{P:metric}, there exist $\eta_0 \in[\varphi(\tilde{\eps},\tilde{\delta}),1]$ and $x \in F_N(p,\tilde{\delta}(\eta_0))$ such that 
\begin{equation}\label{e:quadraticdiff}
\forall y \in F_N(p,\eta_0) \left(d^2(u,x)\leq d^2(u,y)+\tilde{\eps} \right).
\end{equation}
Define $\eta_1:=\frac{\eta_0^2}{24N} \in [\Phi(\eps,\delta),1]$. Since $x  \in F_N(p,\eta_1)$, for any $y \in F_N(p,\eta_1)$,  by Lemma~\ref{L:2.3K} (applied to $T$ and $U$ restricted to $C \cap \overline{B}_N(p)$ and $b=2N$)
\begin{equation*}
\forall t \in [0,1] \left(w_t \in F_N(p, \eta_0) \right),
\end{equation*}
with $w_t:=(1-t)x \oplus ty$. Hence, by \eqref{e:quadraticdiff}
\begin{equation*}
\forall t \in [0,1] \left(d^2(u,x)\leq d^2(u,w_t)+ \frac{\eps^2}{4N^2} \right)
\end{equation*}
and the result follows from Lemma~\ref{L:2.7K} (with $b=2N$) and the fact that $d(x,T(x)),d(x,U(x))\leq \delta(\eta_1)$.
\end{proof}
We now apply, in the setting of $\CAT$ spaces, a technique which corresponds to the quantitative removal of sequential weak compactness in Hilbert spaces -- see \cite{FFLLPP(19)}. 
\begin{proposition}\label{p:removalswc}
Given $u \in C$, let $N \in \N \setminus\{0\}$ be such that $N\geq \max\{d(x_0,p),2 d(u,p)\}$ for some $p \in F$. Let $\rho$ be a monotone rate of asymptotic regularity for both $U$ and $T$. For any
$\eps >0$ and function $\Delta : \N \to (0, 1]$, there exist $n\leq \Psi[N,\rho](\eps,\Delta)$ and $x \in F_N(p,\Delta(n))$ such that
\begin{equation*}
\forall m \geq n \left(\langle \overrightarrow{ux},\overrightarrow{x_mx}\rangle \leq \eps \right), 
\end{equation*}
where $\Psi[N,\rho](\eps,\Delta):=\rho(\Phi(\eps,\widebar{\Delta}))$, with $\widebar{\Delta}(\eta):=\Delta(\rho (\eta))$ for all $\eta \in (0,1]$, and $\Phi:=\Phi[N]$ is as in Proposition~\ref{p:innerproduct}.
\end{proposition}
\begin{proof}
Let $\eps >0$ and $\Delta : \N \to (0, 1]$ be given. By Proposition~\ref{p:innerproduct}, there exist $\eta_0 \in[\Phi(\eps,\widebar{\Delta}),1]$ and $x \in F_N(p,\widebar{\Delta}(\eta_0))$ such that 
\begin{equation}\label{e:innerprod}
\forall y \in F_N(p,\eta_0) \left(\langle \overrightarrow{ux},\overrightarrow{yx}\rangle \leq \eps \right). 
\end{equation}
Let $n:= \rho(\eta_0)$. Observe that $d(x,T(x)),d(x,U(x))\leq \widebar{\Delta}(\eta_0)=\Delta(n)$ and using the monotonicity of $\rho$, we have $n\leq  \Psi[N,\rho](\eps,\Delta)$. Since $d(x_n,p) \leq N$ by Lemma~\ref{l:bounded}, and $\forall m \geq n \left(d(x_n,T(x_n)),d(x_n,U(x_n))\leq \eta_0 \right)$, the result follows from \eqref{e:innerprod}. 
\end{proof}
\section{Metastability}\label{s:convergence}
In this section we prove metastability results pertaining to the iteration $(x_n)$ and extract some consequences of that property. We also consider an iteration \eqref{e:MannHalpernerrors} which is a generalization of the iteration \eqref{e:MannHalpern} with error terms. We show that rates of metastability for a sequence generated by \eqref{e:MannHalpernerrors} follow from the metastability of $(x_n)$.
\begin{theorem}\label{t:main_meta}
	Let $X$ be a $\CAT$ space and $C \subseteq X$ a nonempty convex subset. Consider nonexpansive mappings $T,U:C \to C$ such that  $F:=\Fix(T) \cap \Fix(U) \neq \emptyset$. Let $(\alpha_n),(\beta_n) \subset [0,1]$ and $u,x_0 \in C$.  Assume that we have monotone functions $\Gamma_1, \Gamma_2, \Gamma_3, \Gamma_4$ and a constant $\gamma\in (0,1)$  satisfying $(Q_{\ref{Q1}})-(Q_{\ref{Q5}})$, respectively. Let $N\in\N\setminus\{0\}$ be a natural number such that $N\geq \max\{d(x_0,p),2 d(u,p)\}$, for some $p\in F$.
	
	Then $(x_n)$ generated by \eqref{e:MannHalpern} is a Cauchy sequence with rate of metastability
	\begin{equation*}
		\mu(\eps,f):=2\max\left\{\Sigma\left(\Psi\left(\frac{\tilde{\eps}}{24}, \Delta_{\eps,f}\right)\right), \theta_4\left(\frac{\eps}{4}\right) \right\} +1,
	\end{equation*} 
	where $\mu:=\mu[N,\gamma,\Gamma_1, \Gamma_2, \Gamma_3, \Gamma_4]$ and
	\begin{enumerate}
		\item[] $\tilde{\eps}:=\frac{\eps^2}{16}$,
		\item[] $\Delta_{\eps,f}(n):=\min \left\{\frac{\tilde{\eps}}{30N(P(n)+1)},1\right\}$,
		\item[] $P(n):=f\left(2\max\left\{\Sigma(n), \theta_4\left(\frac{\eps}{4}\right) \right\} +1\right)$,
		\item[] $\Sigma(n):=\sigma\left[\Gamma_2, 4N^2\right]\left(\tilde{\eps}, K(n)\right),\, \text{ with }\, \sigma\, \text{ as defined in Lemma~\ref{L:xu_seq_reals_qt3}}$,
		\item[] $K(n):=\max\left\{n, \rho_3\left(\frac{\tilde{\eps}}{36N}\right)\right\}$,
		\item[] $\theta_4$ as in Lemma~\ref{l:artheta}, $\rho_3$ as in Proposition~\ref{p:asymptoticregularity},
		\item[] $\Psi=\Psi[N,\rho]$ as in Proposition~\ref{p:removalswc}, with $\rho$ as in Lemma~\ref{l:rho}.
	\end{enumerate}
\end{theorem}
\begin{proof}
	Let $\eps>0$ and a function $f:\N\to\N$ be given. By Proposition~\ref{p:removalswc} we may consider $n_0\leq \Psi(\tilde{\eps}/24, \Delta_{\eps,f})$ and $\tilde{x}\in F_N(p, \Delta_{\eps,f}(n_0))$ such that
	\begin{equation*}
		\forall n \geq n_0 \left(\langle \overrightarrow{u\tilde{x}},\overrightarrow{x_n\tilde{x}}\rangle \leq \frac{\tilde{\eps}}{24} \right). 
	\end{equation*}
	Using (W1) and the nonexpansivity of $U$, we have for all $n\in \N$
	\begin{equation*}
		d(x_{2n+2}, \tilde{x})\leq (1-\beta_{n})d(U(x_{2n+1}),\tilde{x}) + \beta_{n}d(x_{2n+1},\tilde{x})\leq d(x_{2n+1},\tilde{x}) + d(U(\tilde{x}), \tilde{x}).
	\end{equation*}
	So, using Lemma~\ref{l:binomial}, we have
	\begin{align}\label{e:dquadrado}
		d^2(x_{2n+2}, \tilde{x})&\leq \left(d(x_{2n+1}, \tilde{x})+d(U(\tilde{x}), \tilde{x})\right)^2 \nonumber\\
		&=d^2(x_{2n+1}, \tilde{x}) + d(U(\tilde{x}), \tilde{x})\left(d(U(\tilde{x}), \tilde{x}) + 2d(x_{2n+1}, \tilde{x})\right) \nonumber\\
		&\leq (1-\alpha_n)^2d^2(T(x_{2n}), \tilde{x})+2\alpha_n(1-\alpha_n)\langle \overrightarrow{T(x_{2n})\tilde{x}},\overrightarrow{u\tilde{x}}\rangle + \alpha_n^2d^2(u, \tilde{x}) + 5N\Delta_{\eps,f}(n_0) \nonumber\\
		&\leq (1-\alpha_n)d^2(x_{2n}, \tilde{x})+\alpha_n\left(2\langle \overrightarrow{u\tilde{x}}, \overrightarrow{T(x_{2n})\tilde{x}}\rangle + 3N^2\alpha_n\right) + 10N\Delta_{\eps,f}(n_0),
	\end{align}
	since $2d^2(u, \tilde{x})\leq (d(u,p)+d(\tilde{x},p))^2\leq (3N/2)^2\leq 3N^2$ and
	\SE{
		d^2(T(x_{2n}),\tilde{x})&\leq d^2(T(x_{2n}),T(\tilde{x}))+d(T(\tilde{x}),\tilde{x})\left(d(T(\tilde{x}),\tilde{x})+ 2d(T(x_{2n}),T(\tilde{x})) \right)\\
		&\leq d^2(x_{2n},\tilde{x})+ \Delta_{\eps,f}(n_0)(1+2\cdot2N) \leq d^2(x_{2n},\tilde{x})+ 5N\Delta_{\eps,f}(n_0).
	}
	Thus, the inequality \eqref{e:dquadrado} can we written in the compact form
	\[
	s_{n+1}\leq (1-\alpha_n)s_n+ \alpha_nr_n+ \mathcal{E},
	\]
	where $s_n:=d^2(x_{2n}, \tilde{x})$, $r_n:=2\langle \overrightarrow{u\tilde{x}}, \overrightarrow{T(x_{2n})\tilde{x}}\rangle + 3N^2\alpha_n$ and $\mathcal{E}:=10N\Delta_{\eps,f}(n_0)$.  Using Cauchy-Schwarz, for $n\geq \max\{n_0, \rho_3(\frac{\tilde{\eps}}{36N})\}$, we have
	\begin{align*}
		2\langle \overrightarrow{u\tilde{x}}, \overrightarrow{T(x_{2n})\tilde{x}}\rangle &=2\langle \overrightarrow{u\tilde{x}}, \overrightarrow{T(x_{2n})x_{2n}}\rangle +2\langle \overrightarrow{u\tilde{x}}, \overrightarrow{x_{2n}\tilde{x}}\rangle\\
		&\leq 2d(u, \tilde{x})d(T(x_{2n}), x_{2n}) + 2\langle \overrightarrow{u\tilde{x}}, \overrightarrow{x_{2n}\tilde{x}}\rangle\\
		&\leq 3N\frac{\tilde{\eps}}{36N} + \frac{2\tilde{\eps}}{24}=\frac{\tilde{\eps}}{6}.
	\end{align*}
	
	By the definition of $\rho_3$, note that $\max\{n_0, \rho_3(\frac{\tilde{\eps}}{36N})\}\geq \Gamma_1(\frac{\tilde{\eps}}{18N^2})$ and so
	\[
	3N^2\alpha_n\leq \frac{\tilde{\eps}}{6}.
	\]
	
	Hence, for all $n\geq K(n_0)$, $r_n\leq \frac{\tilde{\eps}}{3}$. Since
	\[
	\mathcal{E}\leq \frac{10N\tilde{\eps}}{30N(P(n_0)+1)}=\frac{\tilde{\eps}}{3(P(n_0)+1)},
	\]
	we can apply Lemma~\ref{L:xu_seq_reals_qt3} to conclude
	\begin{equation}\label{e:ConclusionXu}
		\forall n\in [\Sigma(n_0), P(n_0)] \left( d(x_{2n}, \tilde{x})\leq \frac{\eps}{4}\right).
	\end{equation}
	Since $d(x_{2n+1}, \tilde{x})\leq d(x_{2n+1}, x_{2n})+d(x_{2n}, \tilde{x})$, with $n_1:= \max\{\Sigma(n_0), \theta_4(\eps/4)\}$ we obtain
	\[
	\forall n\in [n_1, P(n_0)] \left( d(x_{2n}, \tilde{x}),d(x_{2n+1}, \tilde{x})\leq \frac{\eps}{2}\right).
	\]
	The latter entails that for $n\in [2n_1+1, P(n_0)]$, $d(x_n, \tilde{x})\leq \frac{\eps}{2}$.
	One easily sees that the function $\Sigma$ is monotone, and therefore $2n_1+1\leq \mu(\eps, f)$.
	Finally, noticing that $P(n_0)= f(2n_1+1)$, by the triangle inequality we conclude that $\mu$ is a rate of metastability for $(x_n)$.
\end{proof}
\begin{remark}\label{R:tomain}
	\begin{enumerate}
		\item	If $\Gamma_2$ is replaced by a monotone function $\Gamma_2':\N\times (0,+\infty)\to\N$ satisfying
		\[
		\forall \eps >0 \, \forall m\in \N\, \left( \prod_{i=m}^{\Gamma_2'(m,\eps)}(1-\alpha_i)\leq \eps\right),
		\]
		then one can compute a rate of metastability as before but applying Lemma~\ref{L:xu_seq_reals_qt4} instead of Lemma~\ref{L:xu_seq_reals_qt3}. Indeed the only difference would be in the definition of the number function $\Sigma$, which now would use $\sigma'$ instead of $\sigma$,
		\[
		\Sigma(n):=\sigma'\left[\Gamma_2', 4N^2\right]\left(\tilde{\eps}, K(n)\right).
		\]
		\item	The function $\rho_3$ was only required as a rate of convergence towards zero for $(d(T(x_{2n}),x_{2n}))$. It is then clear that $\rho_3$ could be replaced by $\eps\mapsto \max\left\{ \theta_4\left(\frac{\eps}{6}\right), \Gamma_1\left(\frac{\eps}{4N}\right) \right\}$ (cf.\ the end of the proof of Proposition~\ref{p:asymptoticregularity}). 
	\end{enumerate}
\end{remark}
As a corollary to the proof of Theorem~\ref{t:main_meta}, we can derive quantitative information on the Halpern iteration. Previous rates of metastability were obtained by Kohlenbach and Leu\c{s}tean in \cite{KL(12)} using different methods. Their results follow from a quantitative analysis of a proof by Saejung~\cite{Saejung(10)} and rely on a technique to eliminate the use of Banach limits needed in the original proof. 
\begin{corollary}\label{C:Saejungmeta}
	Let $T:C \to C$ be a nonexpansive mapping, $(\alpha_n)\subset [0,1]$ and $y_0,u \in C$. Let $(y_n)$ be a sequence generated by the Halpern schema \eqref{e:Halpern}
	\begin{equation*}
		y_{n+1}:=(1-\alpha_n)T(y_{n}) \oplus \alpha_nu.
	\end{equation*}	
	Assume that we have  monotone functions $\Gamma_1, \Gamma_2, \Gamma_3$ satisfying $(Q_{\ref{Q1}})-(Q_{\ref{Q3}})$, respectively. Let $N\in\N\setminus\{0\}$ be a natural number such that $N\geq \max\{d(y_0,p),2 d(u,p)\}$, for some $p\in \Fix(T)$.
	Then 
	\begin{enumerate}
		\item[$(i)$] $(y_n)$ is asymptotically regular, and we have the following rates of asymptotic regularity
		\begin{enumerate}
			\item[] $\lim d(y_{n+1}, y_n)=0$ with rate of convergence $\widehat{\theta}[{\rm A},{\rm V},D]$,
			\item[] $\lim d(T(y_n),y_n)=0$ with rate of convergence $\tilde{\rho}(\eps):=\max\left\{\widehat{\theta}[{\rm A},{\rm V},D](\frac{\eps}{2}),\Gamma_1\left(\frac{\eps}{3N}\right) \right\}$,
		\end{enumerate}
		where $\widehat{\theta}$ is as in Lemma~\ref{L:xu_seq_reals_qt1}(1), with parameters ${\rm A}(k):=\Gamma_2(k+1)$, ${\rm V}(\eps):=\Gamma_3\left(\frac{2\eps}{3N}\right)$, and $D:=2N$.
		\item[$(ii)$] $(y_n)$ is a Cauchy sequence with rate of metastability
		\begin{equation*}
			\zeta(\eps,f):=\widetilde{\Sigma}\left(\Psi\left(\frac{\tilde{\eps}}{24}, \widetilde{\Delta}_{\eps,f}\right)\right),
		\end{equation*}
		where \begin{enumerate}
			\item[] $\tilde{\eps}:=\frac{\eps^2}{4}$,
			\item[] $\widetilde{\Delta}_{\eps,f}:=\min \left\{\dfrac{\tilde{\eps}}{30N\left(f\left(\widetilde{\Sigma}(n) \right)+1\right)},1\right\}$,
			\item[] $\widetilde{\Sigma}(n):=\sigma\left[\Gamma_3, 4N^2\right]\left(\tilde{\eps}, \widetilde{K}(n)\right),\, \text{ with }\, \sigma\, \text{ as defined in Lemma~\ref{L:xu_seq_reals_qt3}}$,
			\item[] $\widetilde{K}(n):=\max\left\{n, \tilde{\rho}\left(\frac{\tilde{\eps}}{36N}\right)\right\}$,
			\item[] $\Psi=\Psi[N,\tilde{\rho}]\, \text{ as in Proposition~\ref{p:removalswc}}$.
		\end{enumerate}
	\end{enumerate}
\end{corollary}
\begin{proof}	
	To prove part $(i)$, we simplify the arguments in the proof of Lemma~\ref{l:artheta}. As before, one easily proves by induction that $d(y_n,p)\leq N$, for all $n\in\N$. Using (W2) and (W4) one shows that for all $n \in \N$
	\begin{equation*}
		d(y_{n+2},y_{n+1})\leq (1-\alpha_{n+1})d(y_{n+1},y_{n})+\frac{3N}{2}|\alpha_{n+1}-\alpha_n|.
	\end{equation*}
	By an application of Lemma~\ref{L:xu_seq_reals_qt1}(1) we have that $\lim d(y_{n+1},y_{n})=0$ with rate of convergence $\widehat{\theta}[A,V,D]$. Since, using (W1) and the assumption on $N$, for all $n \in \N$ 
	\begin{equation*}
		d(T(y_{n}),y_{n})\leq \frac{3N}{2}\alpha_n+d(y_{n+1},y_{n}),
	\end{equation*}
	we see that $\tilde{\rho}$ is a rate of convergence for $\lim d(T(y_n),y_n)=0$, and conclude the proof of part $(i)$.

	We now argue the metastability property of $(y_n)$.	First consider $(x_n)$ generated by \eqref{e:MannHalpern} with parameters $(\alpha_n, T,\beta_n, U, x_0, u):=(\alpha_n, T, 1/2, \Id_C, y_0, u)$. By an easy induction one sees that $y_{n}=x_{2n}$, for all $n \in \N$. In light of Remark~\ref{R:tomain}(2) and the fact that $\tilde{\rho}$ is a rate of convergence towards zero for $(d(T(y_n),y_n))\equiv (d(T(x_{2n}),x_{2n}))$, we can follow the proof of Theorem~\ref{t:main_meta} (with $\tilde{\eps}:=\eps^2/4$) until \eqref{e:ConclusionXu}, to conclude that
	\[
	\forall n\in \left[\widetilde{\Sigma}(n_0), f\left(\widetilde{\Sigma}(n_0)\right)\right] \left( d(y_{n}, \tilde{x})\leq \frac{\eps}{2}\right).
	\]
	The result then follows using the triangular inequality.
\end{proof}
The next results take into account the presence of error terms in the inductive construction of the iteration \eqref{e:MannHalpern}. We start with an easy lemma which is well-known (see e.g.\ \cite[Proposition~7]{LLPP(21)}). We include its short proof for the sake of completeness.
\begin{lemma}\label{l:Cauchymetric}
	Consider two sequences $(w_n)$ and $(z_n)$ in a metric space. If $(w_n)$ is a Cauchy sequence with rate of metastability $\tau$ and $\lim d(w_n,z_n)=0$ with rate of convergence $\nu$, then $(z_n)$ is also a Cauchy sequence with rate of metastability 
	\begin{equation*}
		\tau_{\nu}(\eps,f):=\max\left\{\tau\left(\frac{\eps}{3},f_{\eps,\nu}\right), \nu\left(\frac{\eps}{3}\right)\right\},      
	\end{equation*}
	where $f_{\eps,\nu}:=f(\max\{n,\nu(\eps/3)\})$, for all $n \in \N$.
\end{lemma}
\begin{proof}
	Let $\eps >0$ and $f:\N \to \N$ be given. By the assumption on $\tau$, consider $n_0 \leq \tau (\eps/3,f_{\eps,\nu})$ such that $d(w_i,w_j)\leq \eps/3$, for all $i,j \in [n_0,f_{\eps,\nu}(n_0)]$. Define $n:=\max\{n_0,\nu(\eps/3)\} \leq \tau_{\nu}(\eps,f)$. Then, for all $i,j \in [n,f_{\eps,\nu}(n_0)]=[n,f(n)]$ we have 
	\SE{
		d(z_i,z_j) \leq d(z_i,w_i)+d(w_i,w_j)+d(w_j,z_j)\leq \eps,
	}
	using the assumption on $\nu$ and the fact that $n \geq n_0$.
\end{proof}
\begin{theorem}\label{t:quantresults1}
	Let $T,U:C \to C$ be nonexpansive mappings and $(\alpha_n),(\beta_n)\subset [0,1]$. Let $(x_n)$ be generated by \eqref{e:MannHalpern} and assume that $(Q_{\ref{Q2}})$ holds with rate of divergence $\Gamma_2$. Consider a sequence $(x'_n)$ satisfying $x'_0=x_0$ and
	\begin{equation}\label{e:MannHalpernerrors}\tag{HM$_e$}
		\begin{cases}
			d(x'_{2n+1},(1-\alpha_n)T(x'_{2n})\oplus \alpha_n u)&\leq \delta_{2n}\\
			d(x'_{2n+2},(1-\beta_n)U(x'_{2n+1})\oplus\beta_nx'_{2n+1})&\leq \delta_{2n+1}, 
		\end{cases}
	\end{equation}
	where $(\delta_n) \subset [0,+\infty)$.
	If  $\sum \delta_{n} <\infty$, or $\lim \frac{\delta_{2n}+\delta_{2n+1}}{\alpha_n}=0$ (in the case $\alpha_n>0$), then $\lim d(x'_n,x_n)=0$. Moreover,
	\begin{enumerate}[$(1)$]
		\item If $\chi_1$ is a Cauchy rate for $\left(\sum \delta_{n} \right)$, then $\widehat{\nu}$ is a rate of convergence for $\lim d(x'_n,x_n)=0$, where
		\begin{equation*}
			\widehat{\nu}(\eps):=2\max\left\{\theta\left(\frac{\eps}{2}\right),\chi_1\left(\frac{\eps}{2}\right)\right\}+3,
		\end{equation*}
		with $\theta:=\widehat{\theta}[\Gamma_2,\chi_1,\lceil \sum_{i=0}^{\chi_1(1)}\delta_i\rceil +1]$, where $\widehat{\theta}$ is as in Lemma~\ref{L:xu_seq_reals_qt1}(1).
		\item If $\chi_2$ is a rate of convergence for $\lim \frac{\delta_{2n}+\delta_{2n+1}}{\alpha_n}=0$, then $\widecheck{\nu}$ is a rate of convergence  for $\lim d(x'_n,x_n)=0$, where
		\begin{equation*}
			\widecheck{\nu}(\eps):=2\max\left\{\theta\left(\frac{\eps}{2}\right),\chi_2\left(\frac{\eps}{2}\right)\right\}+1,
		\end{equation*}
		with $\theta:=\widecheck{\theta}\left[\Gamma_2,\chi_2,\left\lceil\max_{i \leq \chi_2(1)} \left\{\frac{\delta_{2i}+\delta_{2i+1}}{\alpha_i},1\right\}\right\rceil \right]$, where $\widecheck{\theta}$ is as in Lemma~\ref{L:xu_seq_reals_qt1}(2).
	\end{enumerate}
\end{theorem}
\begin{proof}
	By the nonexpansivity of $T$ and (W4), we have for all $n\in \N$
	\begin{equation}\label{e:dx'toximpar}
		\begin{split}
			d(x'_{2n+1}, x_{2n+1})& \leq d(x'_{2n+1},(1-\alpha_{n})T(x'_{2n})\oplus \alpha_{n} u) \\
			&\qquad + d((1-\alpha_{n})T(x'_{2n})\oplus \alpha_{n} u,x_{2n+1})\\
			& \leq (1-\alpha_{n})d(x'_{2n},x_{2n})+\delta_{2n}.
		\end{split}
	\end{equation}
	Similarly
	\begin{equation}\label{e:dx'toxpar}
		\begin{split}
			d(x'_{2n+2},x_{2n+2})&\leq d(x'_{2n+2},(1-\beta_{n})U(x'_{2n+1})\oplus \beta_{n} x'_{2n+1}) \\
			&\qquad + d((1-\beta_{n})U(x'_{2n+1})\oplus \beta_{n} x'_{2n+},x_{2n+2})\\
			& \leq d(x'_{2n+1}, x_{2n+1})+ \delta_{2n+1}\\
			& \leq (1-\alpha_{n})d(x'_{2n},x_{2n})+\delta_{2n}+\delta_{2n+1}.
		\end{split}
	\end{equation}
	We then conclude that $\lim d(x'_{2n},x_{2n})=0$ by applying Lemma~\ref{L:Xu} with $r_n:=0$ and $v_n:=\delta_{2n}+\delta_{2n+1}$, if $\sum \delta_n <\infty$, or with $r_n:=\frac{\delta_{2n}+\delta_{2n+1}}{\alpha_n}$ and $v_n:=0$, if $\lim \frac{\delta_{2n}+\delta_{2n+1}}{\alpha_n}=0$. Then by \eqref{e:dx'toximpar}, since $\delta_n \to 0$, also $\lim d(x'_{2n+1},x_{2n+1})=0$ . Hence $\lim d(x'_{n},x_{n})=0$.

	Let us show part (1). Clearly, \eqref{e:dx'toxpar} entails that $d(x'_{2n},x_{2n})\leq  \sum_{i=0}^{2n-1}\delta_i$. With $D:= \lceil \sum_{i=0}^{\chi_1(1)}\delta_i\rceil +1$ it is easy to see that the assumption on $\chi_1$ implies that $D$ is an upper bound on  $\left( \sum \delta_n \right)$, and so also on $\left(d(x'_{2n},x_{2n})\right)$. 
	The function $\chi_1$ is also a Cauchy rate for $\left(\sum v_n\right)$, where  $v_n:=\delta_{2n}+\delta_{2n+1}$. Indeed, for all $\eps >0$ and $n \in \N$
	\SE{
		\sum_{i=\chi_1(\eps)+1}^{\chi_1(\eps)+n}v_n=\sum_{i=\chi_1(\eps)+1}^{\chi_1(\eps)+n}\delta_{2i}+\delta_{2i+1}=\sum_{i=2\chi_1(\eps)+2}^{2\chi_1(\eps)+2n+1}\delta_{i}\leq \sum_{i=\chi_1(\eps)+1}^{2\chi_1(\eps)+2n+1}\delta_{i} \leq \eps.
	}
	Now, by Lemma~\ref{L:xu_seq_reals_qt1}(1) we have that $\theta=\widehat{\theta}[\Gamma_2,\chi_1,D]$ is a rate of convergence for $\lim d(x'_{2n},x_{2n})=0$. Since $\chi_1$ is a Cauchy rate for $\left( \sum \delta_n \right)$, we have $\delta_{2n} \to 0$ with rate of convergence $\chi_1+1$. Using \eqref{e:dx'toximpar} we conclude that 
	\SE{
		\forall n \geq \max\left\{\theta\left(\frac{\eps}{2}\right), \chi_1\left(\frac{\eps}{2}\right)+1\right\} \left(d(x'_{2n},x_{2n}), d(x'_{2n+1},x_{2n+1}) \leq \eps \right),
	}
	which entails part (1).
	
	We now turn to part (2). From the assumption on $\chi_2$, we have that $\max_{i \leq \chi_2(1)} \left\{\frac{\delta_{2i}+\delta_{2i+1}}{\alpha_i},1\right\}$ is an upper bound on the sequence $\left(\frac{\delta_{2n}+\delta_{2n+1}}{\alpha_n} \right)$. Using \eqref{e:dx'toxpar}, one shows by induction that $\left(d(x'_{2n},x_{2n})\right)$ is bounded by $D$, where $D:=\left\lceil\max_{i \leq \chi_2(1)} \left\{\frac{\delta_{2i}+\delta_{2i+1}}{\alpha_i},1\right\}\right\rceil$. Now, by Lemma~\ref{L:xu_seq_reals_qt1}(2) we have that $\theta=\widecheck{\theta}[\Gamma_2,\chi_2,D]$ is a rate of convergence for $\lim d(x'_{2n},x_{2n})=0$. Since $\chi_2$ is a rate of convergence for $\lim \frac{\delta_{2n}+\delta_{2n+1}}{\alpha_n}=0$, we have $\delta_{2n} \to 0$ with the same rate of convergence. Using \eqref{e:dx'toximpar} we conclude that 
	\SE{
		\forall n \geq \max\left\{\theta\left(\frac{\eps}{2}\right), \chi_2\left(\frac{\eps}{2}\right)\right\} \left(d(x'_{2n},x_{2n}), d(x'_{2n+1},x_{2n+1}) \leq \eps \right),
	}
	which entails the result.
\end{proof}
A rate of metastability for the generalized sequence with error terms \eqref{e:MannHalpernerrors} can now be immediately obtained from Theorem~\ref{t:main_meta} using Lemma~\ref{l:Cauchymetric} and Theorem~\ref{t:quantresults1}. Additionally, it is clear that $(x'_n)$ is asymptotically regular and rates of asymptotic regularity can be trivially computed using the rates in Proposition~\ref{p:asymptoticregularity} together with the rates of convergence in Theorem~\ref{t:quantresults1}.

\begin{corollary}
	Consider sequences $(\alpha_n)$, $(\beta_n) \subset [0,1]$, a sequence of error terms $(\delta_n) \subset [0,+\infty)$ satisfying $\sum \delta_{n} < \infty$, or $\lim \frac{\delta_{2n}+\delta_{2n+1}}{\alpha_n}=0$ (if $\alpha_n >0$), and let $(x'_n)$ be generated by \eqref{e:MannHalpernerrors}. Then, under the conditions of Theorem~\ref{t:main_meta}, the sequence $(x'_n)$ is a Cauchy sequence. Moreover,
	\begin{enumerate}[$(1)$]
		\item if $\chi_1$ is a Cauchy rate for $\left(\sum \delta_{n} \right)$, then $\mu_{\widehat{\nu}}$ is a rate of metastability for $(x'_n)$;
		\item if $\chi_2$ is a rate of convergence for $\lim \frac{\delta_{2n}+\delta_{2n+1}}{\alpha_n}=0$, then $\mu_{\widecheck{\nu}}$ is a rate of metastability for $(x'_n)$;
	\end{enumerate}
	where $\mu$ is as in Theorem~\ref{t:main_meta}, $\widehat{\nu},\widecheck{\nu}$ are as in Theorem~\ref{t:quantresults1} and the construction $\tau_\nu$ is as in Lemma~\ref{l:Cauchymetric}.
\end{corollary}

\section{Strong Convergence}\label{s:Strongconvergence}
We will now argue how the metastability result from the previous section actually entails the strong convergence of the iteration $(x_n)$. We begin by recalling a well-known characterization of the metric projection in terms of the quasi-linearization function.
\begin{lemma}[\cite{DehghanRooin}]\label{l:charactProj}	Let $S$ be a nonempty convex closed subset of a complete $\CAT$ space $X$. For any $u\in X$, let $P_S(u)$ denote the metric projection of $u\in X$ onto $S$. Then,
	\[
	\forall y\in S\, \left(\langle \overrightarrow{uP_S(u)}, \overrightarrow{yP_S(u)}\rangle \leq 0\right).
	\]
\end{lemma}
We are now ready to give the proof of the main theorem.
\begin{proof}[\textbf{Proof of Theorem~\ref{t:main}}]
	From Theorem~\ref{t:main_meta} and Lemma~\ref{l:metaCauchy}, $(x_n)$ is a Cauchy sequence and therefore convergent since $X$ is complete. Let $z=\lim x_n$, which is an element of $C$, because $C$ is closed. By Proposition~\ref{p:asymptoticregularity} and the continuity of $T$ and $U$, we conclude that $z\in F$. From Lemma~\ref{l:charactProj} (with $S=F$), we conclude that $\langle \overrightarrow{uP_F(u)}, \overrightarrow{zP_F(u)}\rangle \leq 0$.
	Using Lemma~\ref{l:binomial}, with $s_n:=d^2(x_{2n}, P_F(u))$ we have
	\begin{align}
		s_{n+1}&\leq d^2(x_{2n+1}, P_F(u)) \nonumber\\
		&\leq (1-\alpha_n)^2d^2(T(x_{2n}), P_F(u))+2\alpha_n(1-\alpha_n)\langle \overrightarrow{uP_F(u)},\overrightarrow{T(x_{2n})P_F(u)}\rangle + \alpha_n^2d^2(u, P_F(u)) \nonumber\\
		&\leq (1-\alpha_n)d^2(x_{2n}, P_F(u))+\alpha_n\left(2\langle \overrightarrow{uP_F(u)}, \overrightarrow{T(x_{2n})P_F(u)}\rangle + \alpha_nd^2(u, P_F(u))\right)\nonumber\\
		& =  (1-\alpha_n)s_n+\alpha_n r_n,\label{ineqProofmain}
	\end{align}
	where $r_n=2\langle \overrightarrow{uP_F(u)}, \overrightarrow{T(x_{2n})P_F(u)}\rangle + \alpha_nd^2(u, P_F(u))$.
	Since
	\SE{
		\langle \overrightarrow{uP_F(u)}, \overrightarrow{T(x_{2n})P_F(u)}\rangle &=  \langle \overrightarrow{uP_F(u)}, \overrightarrow{T(x_{2n})x_{2n}}\rangle+\langle \overrightarrow{uP_F(u)}, \overrightarrow{x_{2n}z}\rangle+\langle \overrightarrow{uP_F(u)}, \overrightarrow{zP_F(u)}\rangle\\
		&\leq  d(u,P_F(u))\left(d(T(x_{2n}),x_{2n})+d(x_{2n},z)\right) \to 0,
	}
	and $\alpha_n\to 0$, we conclude that $\limsup r_n \leq 0$. Since $\sum \alpha_n=\infty$, we can apply Lemma~\ref{L:Xu} to conclude that $\lim s_n=0$. This means that $\lim x_{2n}=P_F(u)$, and so $z=P_F(u)$.
\end{proof}

As discussed in the Introduction, Theorem~\ref{t:main} is a twofold generalization of the recent strong convergence result, established in the setting of Hilbert spaces, by Bo\c{t}, Csetnek, and  Meier~\cite[Theorem~3]{Botetal(19)}. On the one hand, our iteration is more general than the one considered by Bo\c{t} \emph{et al.} and, on the other hand, the strong convergence is established in $\CAT$ spaces, which are frequently considered the non-linear generalization of Hilbert spaces. Our schema is also a generalization of the algorithm considered by Leus{\c t}ean and Cheval in \cite{CK(ta)} in the setting of hyperbolic spaces. Even though the latter context is more general than that of  $\CAT$ spaces, our Theorem~\ref{t:main} succeeds in proving strong convergence while \cite[Theorems~4.1 and 4.2]{CK(ta)} only show the asymptotic regularity of the iteration. Observe that not only the restriction to $\CAT$ spaces was required to prove strong convergence, but we also needed to consider the condition $0<\liminf \beta_n$ (which has no corresponding condition in \cite{Botetal(19)} nor in \cite{CK(ta)}) in order to obtain asymptotic regularity for the more general iteration \eqref{e:MannHalpern}. In light of the new results in \cite{CKL(ta)} connecting the iteration \eqref{CL} with the modified Halpern iteration, our results also provides quantitative information on the modified Halpern iteration (first obtained in \cite{SK}). 

Since the Halpern iteration follows from a particular case of \eqref{e:MannHalpern} -- cf.\ the proof of Corollary~\ref{C:Saejungmeta} --, Theorem~\ref{t:main} also allows to recover Saejung's strong convergence of the Halpern iteration in $\CAT$ spaces from \cite[Theorem~2.3]{Saejung(10)}. We would like to point out that, contrary to Saejung's proof, we do not require the heavy machinery of Banach limits. Finally, we point out that Theorem~\ref{t:quantresults1}, also extends the strong convergence result to the iterative schema \eqref{e:MannHalpernerrors}.

\section{Forward-backward and Douglas-Rachford algorithms}\label{s:FBDR}

In this section we work in a Hilbert space $H$ and apply our main results to obtain strongly convergent variants of the Forward-backward and Douglas-Rachford algorithms.

 We recall that a multi-valued operator $\mathsf{T}:H \rightrightarrows H$  is \emph{monotone} if whenever $(x,y)$ and $(x',y')$ are elements of the graph of $\mathsf{T}$, it holds that $\langle x-x',y-y'\rangle \geq 0$. A monotone operator $\mathsf{T}$ is said to be \emph{maximal monotone} if  the graph of $\mathsf{T}$ is not properly contained in the graph of any other monotone operator on $H$. 

We use $J_{\mathsf{T}}$  to denote the \emph{resolvent function} of $\mathsf{T}$, i.e.\ the single-valued function defined by $J_{\mathsf{T}} = (I + \mathsf{T} )^{-1}$ and $R_{ \mathsf{T}}$ to denote the nonexpansive  \emph{reflected resolvent function} defined by $R_{ \mathsf{T}}:=2J_{\mathsf{T}} - \Id$.

\begin{definition}
A mapping $U : H \to H$ is called \emph{firmly nonexpansive} if
$$\forall x, y \in H\left( \norm{U(x)-U(y)}^2 \leq \norm{x-y}^2-\norm{(\Id -U)(x)- (\Id-U)(y)}^2 \right).$$
\end{definition}
Clearly, if $U$ is firmly nonexpansive then it is nonexpansive. 
 For $c>0$, the resolvent function $J_{c\mathsf{T}}$ is firmly nonexpansive and the set of fixed points of $J_{c\mathsf{T}}$ coincides with the set of all zeros of $\mathsf{T}$. More information on firmly nonexpansive mappings can be found in \cite{R(77),GR(84)}.

For $\alpha \in (0,1]$, a function $U:H \to H$ is called \emph{$\alpha$-averaged}\footnote{The standard definition \cite{BBR(78)} asks for $\alpha \in (0,1)$. With this extension, $1$-averaged is just another way of saying nonexpansive.} if there exists a nonexpansive operator $U':H \to H$ such that $U=(1-\alpha)\Id+\alpha U'$. The $\alpha$-averaged operators are always nonexpansive. Moreover, the $\frac{1}{2}$-averaged operators coincide with the firmly nonexpansive operators.  

In the following corollaries we assume that we have functions $\Gamma_1, \Gamma_2, \Gamma_3, \Gamma_4$ satisfying conditions $(Q_{\ref{Q1}})-(Q_{\ref{Q4}})$.
\begin{corollary}\label{c:Cor1}
Let $\alpha \in (0,1]$ and $U:H \to H$ be $\alpha$-averaged. Given $(\alpha_n) \subset [0,1]$, $(\beta_n)\subset [1-\frac{1}{\alpha},1]$, and $x_0,u \in H$, consider $(x_n)$ generated by 
\begin{equation}\label{e:HMaveraged}\tag{H$_{\Id}$M}
\begin{cases}
x_{2n+1}&=(1-\alpha_n)x_{2n}+\alpha_nu\\
x_{2n+2}&=(1-\beta_n)U(x_{2n+1})+ \beta_n x_{2n+1}.
\end{cases}
\end{equation}
 Let $\sigma \in (0,1)$ be such that $\alpha\geq \sigma$.  
Assume that there exist $N \in \N \setminus\{0\}$ such that $N \geq \max\{\norm{x_0-p},2\norm{u-p}\}$, for some $p \in \Fix U$ and $\gamma \in (0,\frac{1}{2\sigma}]$ satisfying $1-\frac{1}{\alpha}+\gamma\leq \beta_n \leq 1-\gamma$, for all $n \in \N$.

Then $(x_n)$ converges strongly to $P_{\Fix(U)}(u)$ and 
\begin{equation*}
\mu_1[N,\sigma,\gamma,\Gamma_1, \Gamma_2, \Gamma_3, \Gamma_4]:=\mu[N,\sigma\gamma,\Gamma_1, \Gamma_2, \Gamma_3, \Gamma_4], \mbox{ with $\mu$ as in Theorem~\ref{t:main_meta},}
\end{equation*} 
is a rate of metastability for $(x_n)$.
\end{corollary}
\begin{proof}
Since $U$ is $\alpha$-averaged, there exists $U'$ nonexpansive and such that $U=(1-\alpha)\Id+\alpha U'$. For all $n \in \N$
\begin{equation*}
\begin{split}
x_{2n+2}&=(1-\beta_n)U(x_{2n+1})+ \beta_n x_{2n+1}\\
&=(1-\beta_n)[(1-\alpha)\Id+\alpha U'](x_{2n+1})+ \beta_n x_{2n+1}\\
&=(1-\tilde{\beta}_n)U'(x_{2n+1})+ \tilde{\beta}_n x_{2n+1},
\end{split}
\end{equation*}
with $\tilde{\beta}_n:=1- \alpha+\alpha\beta_n$.
Hence $(x_n)$ is generated by \eqref{e:MannHalpern}, using the sequences $(\alpha_n),(\tilde{\beta}_n) \subset [0,1]$, $T=\Id$ and $U= U'$. It is easy to see that $\sigma\gamma$ satisfies condition $(Q_{\ref{Q5}})$ for the sequence $(\tilde{\beta}_n)$. Note that $\Fix U = \Fix U'=F$. Since $\alpha \in (0,1]$, for all $\eps >0$ and $n \in \N$ we have
 $$\sum_{i=\Gamma_4(\eps)+1}^{\Gamma_4(\eps)+n}|\tilde{\beta}_{i+1}-\tilde{\beta}_{i}|\leq \sum_{i=\Gamma_4(\eps)+1}^{\Gamma_4(\eps)+n}|\beta_{i+1}-\beta_{i}|\leq\eps,$$
and so condition $(Q_{\ref{Q4}})$ still holds with $\Gamma_4$ for the sequence $(\tilde{\beta}_n)$.
Hence, by Theorem~\ref{t:main} the sequence $(x_n)$ converges strongly to $P_{\Fix(U)}(u)$. The rate of metastability follows from an application of Theorem~\ref{t:main_meta} with $\gamma:=\sigma\gamma$.
\end{proof}

For $\delta >0$, a function $U:H \to H$ is said to be \emph{$\delta$-cocoercive} if for all $x,y \in H$,
$$\langle x-y, U(x)-U(y)\rangle \geq \delta \norm{U(x)-U(y)}^2,$$
which is equivalent to say that $\delta U$ is firmly nonexpansive. In the following result we give a strongly convergent version of the forward-backward algorithm which extends \cite[Theorem~7]{Botetal(19)} (as well as the quantitative analysis from \cite[Corollary~2]{DP(ta)}). The algorithm \eqref{e:GFB} below can be seen as a generalization in the sense that for $\alpha_n=\beta_n\equiv 0$ the sequence of even terms $(x_{2n})$ is indeed the original weakly convergent forward-backward algorithm by Lions and Mercier \cite{LM(79)}.

\begin{corollary}\label{c:Cor2}
Let $U_1: H \rightrightarrows H$ be maximal monotone  and $U_2:H \to H$ be $\delta$-cocoercive, for some $\delta >0$. Let $c \in (0,2\delta]$. Given $(\alpha_n) \subset [0,1]$, $(\beta_n) \subset [1-\frac{4\delta-c}{2\delta},1]$, and $x_0, u \in H$, consider $(x_n)$ generated by 
\begin{equation}\tag{GFB}\label{e:GFB}
\begin{cases}
	x_{2n+1}&=(1-\alpha_n)x_{2n}+\alpha_nu\\
	x_{2n+2}&=(1-\beta_n)J_{cU_1}\left(x_{2n+1}-cU_2(x_{2n+1})\right)+ \beta_n x_{2n+1}.
\end{cases}
\end{equation}

Assume that there exist $N \in \N \setminus\{0\}$ such that $N \geq \max\{\norm{x_0-p},2\norm{u-p}\}$, for some $p \in zer (U_1+U_2)$ and $\gamma \in(0,1]$ satisfying $1-\frac{4\delta-c}{2\delta}+\gamma\leq \beta_n\leq 1-\gamma$, for all $n\in \N$. Then $(x_n)$ converges strongly to $P_{zer(U_1+U_2)}(u)$ and
$$
\mu_2:=\mu_2[N,\gamma,\Gamma_1, \Gamma_2, \Gamma_3, \Gamma_4]:=\mu[N,\gamma/2,\Gamma_1, \Gamma_2, \Gamma_3, \Gamma_4], \text{ with $\mu$ as in Theorem~\ref{t:main_meta}},
$$
is a rate of metastability for $(x_n)$.
\end{corollary}

\begin{proof}
It is straightforward to see that the iteration $(x_n)$ is generated by \eqref{e:HMaveraged} with $U:=J_{c U_1} \circ (\Id -c U_2)$. The resolvent function $J_{c U_1}$ is firmly nonexpansive, i.e.\ $\frac{1}{2}$-averaged. Furthermore, we have $\frac{c}{2\delta} \in (0,1]$ and so it follows from \cite[Proposition~4.39]{BC(17)} that $(\Id -c U_2)$ is $\frac{c}{2\delta}$-averaged. If $c < 2\delta$, then we use \cite[Theorem~3$(b)$]{OY(02)} to conclude that $U$ is $\frac{2\delta}{4\delta-c}$-averaged. If $c = 2\delta$, the fact that $U_2$ is $\delta$-cocoercive entails that $\Id - c U_2$ is nonexpansive. Hence $U$ is nonexpansive, or equivalently (following the extended notion of 1-averaged), $U$ is $\frac{2\delta}{4\delta-c}$-averaged since $\frac{2 \delta}{4 \delta - c}=1$.
Since $\Fix U=zer(U_1+U_2)$ \cite[Proposition~26.1$(iv)(a)$]{BC(17)}, we have $N \geq \max\{\norm{x_0-p},\norm{u-p}\}$, for some $p \in \Fix U$. From the fact that $\frac{2\delta}{4\delta-c} \geq \frac{1}{2}$, we may apply Corollary~\ref{c:Cor1} with $\sigma:=\frac{1}{2}$ to conclude the result.
\end{proof}

\begin{remark}
	The fact that no quantitative information regarding `$c>0$' is present in the rate of metastability reflects the fact that this condition is only needed to make sense of the definition of the resolvent function. In fact, if in Corollary~\ref{c:Cor2} the resolvent function $J_{c U_1}$ is replaced by an arbitrary firmly nonexpansive mapping $J$, and $p$ is some point in $\Fix (J \circ (\Id -c U_2))$, then the result holds also for $c=0$ (immediately by Corollary~\ref{c:Cor1} with $\alpha=1/2$). In such case, the sequence $(\beta_n)$ is allowed to vary in the interval $[-1,1]$.
\end{remark}
%

In the following result we give a strongly convergent version of the Douglas–Rachford algorithm which extends \cite[Theorem~10]{Botetal(19)} (as well as the quantitative analysis from \cite[Corollary~3]{DP(ta)}). The algorithm \eqref{e:GDR} below can be seen as a generalization in the sense that for $\alpha_n\equiv 0$ and $\beta_n \equiv -1$ the sequence of even terms $(x_{2n})$ is indeed the original weakly convergent Douglas-Rachford algorithm \cite{DR(56)}.
\begin{corollary}\label{c:Cor3}
Let $U_1,U_2: H \rightrightarrows H$ be two maximal monotone operators and $c >0$. Given $(\alpha_n) \subset [0,1]$, $(\beta_n) \subset [-1,1]$, and $x_0, u \in H$, consider $(x_n)$ generated by
\begin{equation}\tag{GDR}\label{e:GDR}
\begin{cases}
x_{2n+1}&=(1-\alpha_n)x_{2n}+\alpha_nu\\
\quad y_n&=J_{c U_2}(x_{2n+1})\\
\quad z_n&=J_{c U_1}(2y_n-x_{2n+1})\\
x_{2n+2}&=x_{2n+1}+(1-\beta_n)(z_n-y_n).
\end{cases}
\end{equation} 
Assume that there exist $N \in \N \setminus\{0\}$ such that $N \geq \max\{\norm{x_0-p},2\norm{u-p}\}$ for some $p \in \Fix(R_{cU_1}\circ R_{cU_2})$, and $\gamma \in(0,1]$ satisfying $-1+\gamma\leq \beta_n\leq 1-\gamma$, for all $n\in \N$.
Then, 
\begin{enumerate}[$(i)$]
\item $(x_n)$ converges strongly to $\overline{x}=P_{\Fix (R_{c U_1}\circ R_{c U_2})}(u)$ and
$$\mu_3:=\mu_3[N,\gamma,\Gamma_1, \Gamma_2, \Gamma_3, \Gamma_4]:=\mu[N,\gamma/2,\Gamma_1, \Gamma_2, \Gamma_3, \Gamma_4], \, \text{ with $\mu$ as in Theorem~\ref{t:main_meta}},$$
is a rate of metastability for $(x_n)$;
\item $(y_n)$ and $(z_n)$ converge strongly to $J_{cU_2}(\overline{x})\in zer(U_1+U_2)$ and
\begin{equation*}
		\mu_4(\eps,f):=\mu_3(\eps,2f+1), \qquad \mu_5(\eps,f):=\mu_3(\eps/3,2f+1),
\end{equation*}
are rates of metastability for $(y_n)$, $(z_n)$, respectively.
\end{enumerate}
\end{corollary}

\begin{proof}
It is easy to see that the iteration $(x_n)$ is generated by \eqref{e:HMaveraged} using the sequences $(\alpha_n), (\tilde{\beta}_n) \subset [0,1]$, where $\tilde{\beta}_n:=\frac{1+\beta_n}{2}$ for all $n\in \N$, and with the nonexpansive map $U:=R_{c U_1} \circ R_{c U_2}$. It is clear that we are still in the conditions of Theorem~\ref{t:main}, and so we conclude that $(x_n)$ converges strongly to $\overline{x}=P_{\Fix (R_{c U_1}\circ R_{c U_2})}(u)$.

Since the condition (Q$_{\ref{Q4}}$) still holds for $(\tilde{\beta}_n)$ with the function $\Gamma_4$, and we have for all $n\in \N$, $\frac{\gamma}{2}\leq \tilde{\beta}_n\leq 1-\frac{\gamma}{2}$, we can apply Theorem~\ref{t:main_meta} to obtain the desired rate of metastability. This concludes part $(i)$.

We now prove part $(ii)$. From part $(i)$, the definition of $(y_n)$ and the continuity of $J_{cU_2}$, we conclude that $(y_n)$ converges strongly to $J_{cU_2}(\overline{x})$. By \cite[Proposition 26.1$(iii)(b)$]{BC(17)}, we have $zer(U_1+U_2)=J_{cU_2}[\Fix(U)]$, hence $J_{cU_2}(\overline{x})\in zer(U_1+U_2)$. From the recursive definition of $x_{2n+2}$, we see that
\[
\norm{z_n-y_n}=\frac{1}{1-\beta_n}\norm{x_{2n+2}-x_{2n+1}}\leq \frac{1}{\gamma}\norm{x_{2n+2}-x_{2n+1}}\to 0,
\]
which entails that $(z_n)$ also converges strongly to $J_{cU_2}(\overline{x})$. We now argue the rates of metastability for $(y_n)$ and $(z_n)$. For all $i,j\in\N$, we have
\begin{equation*}
	\norm{y_i-y_j}= \norm{J_{cU_2}(x_{2i+1})-J_{cU_2}(x_{2j+1})} \leq \norm{x_{2i+1}-x_{2j+1}}
\end{equation*}
and
\begin{align*}
\norm{z_i-z_j}&= \norm{J_{cU_1}(2y_i-x_{2i+1})-J_{cU_1}(2y_j-x_{2j+1})}\\
& \leq \norm{2y_i-x_{2i+1}-(2y_j-x_{2j+1})}\\
& \leq 2\norm{y_i-y_j}+\norm{x_{2i+1}-x_{2j+1}}\\
& \leq 3\norm{x_{2i+1}-x_{2j+1}}.	
\end{align*}
The result now follows from part $(i)$ and the fact that if $i\in [n,f(n)]$, then $2i+1\in [n,2f(n)+1]$.
\end{proof}

The algorithms considered by Bo\c{t}, Csetnek and  Meier \cite{Botetal(19)} are the particular case of the algorithms considered in this section when $u=0$. 

\section{Final remarks}\label{s:final}

We consider in this paper an algorithm which mixes Halpern and Krasnoselskii-Mann iterative definitions in an alternating way. Under appropriate conditions we show asymptotic regularity for any sequence $(x_n)$ generated by this algorithm, and construct effective rates of asymptotic regularity. We prove that $(x_n)$ is a Cauchy sequence and obtain a rate of metastability in the sense of Tao \cite{T(08b),T(08a)}. Using this result, we show that our iterative method strongly approximates a common fixed point of two nonexpansive mappings. This strong convergence result allowed us to define strongly convergent versions of the forward-backward and the Douglas-Rachford algorithms. Our results generalize recent work by Bo\c{t}, Csetnek and  Meier \cite{Botetal(19)}, and Cheval and Leu\c{s}tean \cite{CK(ta)}.

In fact, the proof mining program enables, in certain instances, to obtain generalized versions of mathematical results. We briefly comment on three recent examples. The first example concerns the analysis, carried out in \cite{KLAN(21)}, of a proof of the ``Lion-Man'' game whose convergence crucially uses the compactness of the metric space together with its betweenness property. Proof mining allowed to weaken the compactness assumption to a boundedness condition, if \emph{betweenness} is upgraded to \emph{uniform betweenness} (which coincides with betweenness in the compact case anyway), resulting in a striking generalization since now the convergence holds in all bounded subsets of uniformly convex Banach spaces, $\mathrm{CAT}(\kappa)$ spaces ($\kappa>0$, where one anyhow always has a boundedness assumption), etc.\ The second example comes from \cite{KP(22)} where the authors analysed a proof that the strong convergence of the so-called viscosity generalizations of Browder and Halpern-type algorithms can be reduced to the convergence of the original Browder/Halpern algorithms in the setting of Banach spaces. The quantitative analysis allowed to generalize this reduction to the non-linear setting of hyperbolic spaces, which in turn allowed for several applications of previously known rates of metastability for these algorithms. Our final example can be found in \cite{S(ta)} where the author adapts the analysis of the strong convergence of the Halpern type Proximal Point Algorithm, given in \cite{K(20)}, from the setting of Banach spaces to $\CAT$ spaces. This gave rise to new qualitative convergence results. The key observation is that instead of \emph{strong nonexpansivity} and corresponding SNE-moduli, one can also work with \emph{strong quasi-nonexpansivity} and corresponding SQNE-moduli (introduced in \cite{K(16)}). More information on strongly quasi-nonexpansive operators can be found in \cite[Section~2.1]{CRZ(18)}.
 
Finally, let us explain how the proof of Theorem~\ref{t:main} was obtained. A proof of strong convergence can be carried out with the usual arguments in the setting of Hilbert spaces, as explained below. However, it is not clear if such arguments can be generalized to a non-linear setting. Nevertheless, a passage through quantitative results made it possible to overcome this problem. 

Let us elaborate on this matter. A proof in Hilbert spaces is guided by the following steps:
\begin{enumerate}
\item[$(1)$]  \underline{$(x_n)$ is bounded}: Follow the arguments of Lemma~\ref{l:bounded}.
\item[$(2)$] \underline{Asymptotic regularity of $(x_n)$}: Follow the arguments of Lemma~\ref{l:artheta} and Proposition~\ref{p:asymptoticregularity}.
\item[$(3)$] \underline{Projection argument}: With $\tilde{x}$ the projection point onto $F$ of $u$, we have $\forall y \in F \left(\langle u- \tilde{x}, y-\tilde{x}\rangle \leq 0\right)$. 
\item[$(4)$] \underline{Sequential weak compactness and demiclosedness}: Pick a subsequence $(x_{n_j})$ of $(x_{2n})$ such that 
\[
\limsup\, \langle u- \tilde{x}, x_{2n}-\tilde{x}\rangle =\lim_{j \to \infty}\, \langle u- \tilde{x}, x_{n_j}-\tilde{x}\rangle,
\] 
and simultaneously $(x_{n_j})$ converges weakly to some $y \in F$. Here we are using (twice) the following demiclosedness principle.
\begin{lemma}[Demiclosedness principle \cite{B(65)}]
Let $C$ be a closed convex subset of $H$ and let $f : C \to C$ be a nonexpansive mapping such that $\mathrm{Fix}(f)\neq \emptyset$. Assume that $(x_n)$ is a sequence in $C$ such that $(x_n)$ weakly converges to $x \in C$ and  $((\Id_C - f)(x_n))$ converges strongly to $y \in H$. Then $(\Id_C - f)(x) = y$.
\end{lemma}
By step (3) it follows that $\limsup \,\langle u-\tilde{x},x_n-\tilde{x}\rangle \leq 0$.
\item[$(5)$]  \underline{Main combinatorial part}: Following the inequalities culminating in  \eqref{ineqProofmain} (in the proof of Theorem~\ref{t:main}), we are then able to apply Lemma~\ref{L:Xu} and conclude that $x_n \to \tilde{x}$.  
\end{enumerate}

As shown in Section~\ref{s:asymptoticregularity}, steps (1) and (2) are valid in $\CAT$ spaces (actually, even in general UCW spaces \cite{LP(ta)}). It is however not clear how the sequential weak compactness argument could be carried out in this more general context.

On the other hand, the quantitative analysis of this proof in Hilbert spaces is known to be possible. The only roadblocks to the analysis are the projection and the sequential weak compactness arguments.  Yet, recent developments in proof mining provide a method to eliminate such arguments. The solution is twofold. For eliminating the full logical strength of the projection argument one relies on the crucial observation by Kohlenbach \cite{kohlenbach2011quantitative} that a weaker $\varepsilon$-version already suffices to obtain quantitative results. This holds true in both Hilbert and $\CAT$ spaces. As for sequential weak compactness one makes use of the general principle developed in \cite{FFLLPP(19)}. This macro was developed in the general setting of metric spaces aiming to bypass sequential weak compactness arguments in Hilbert spaces. The fundamental observation that the macro can be applied in the setting of $\CAT$ spaces (no longer with concerns about weak compactness) rests on establishing some version of \cite[Proposition~4.3]{FFLLPP(19)}. This can be seen to correspond to our treatment of the metric projection via the quasi-linearization function characterization carried out in Section~\ref{s:Projection} and culminating in Proposition~\ref{p:removalswc} (which is the conclusion of \cite[Proposition~4.3]{FFLLPP(19)}, adapted to two maps, with $\varphi(x,y):=\langle \overrightarrow{ux},\overrightarrow{uy}\rangle$). The argument in Section~\ref{s:Strongconvergence} brings us back to a ``qualitative'' statement, giving rise to our proof of Theorem~\ref{t:main}.  

Finally, we would like to mention some very recently obtained results on asymptotic regularity for particular choices of parameters. For the particular case of the \eqref{CL} iteration (and also the modified Halpern iteration), \emph{linear} rates of asymptotic regularity were obtained  in \cite{CKL(ta)}. The asymptotic regularity of the general iteration \eqref{e:MannHalpern} was studied in the context of UCW spaces in the forthcoming  \cite{LP(ta)}. For a particular choice of parameters it was shown to admit quadratic rates in $\CAT$ spaces.

\section*{Acknowledgements}

 The first author acknowledges the support of FCT - Funda\c{c}\~ao para a Ci\^{e}ncia e Tecnologia under the projects: UIDP/04561/2020 and UIDP/04674/2020, and the research centers CMAFcIO -- Centro de Matem\'{a}tica, Aplica\c{c}\~{o}es Fundamentais e Investiga\c{c}\~{a}o Operacional and CIMA -- Centro de Investigação em Matemática e Aplicações. 

\noindent The second author was supported by the German Science Foundation (DFG Project KO 1737/6-2).

\noindent This work benefited from discussions with Ulrich Kohlenbach. 

\bibliography{References}{}
\bibliographystyle{abbrv}

\end{document}